\documentclass[a4paper]{article}
\usepackage{amsmath,amssymb,a4wide}
\usepackage{color,graphicx,epstopdf}
\usepackage[a4paper]{geometry}

\newcommand{\bu}{\boldsymbol u}
\newcommand{\bv}{\boldsymbol v}

\newcommand{\bU}{\boldsymbol U}

\newcommand{\di}{\mathrm d}

\newtheorem{Theorem}{Theorem}
\newtheorem{lema}{Lemma}
\newcounter{remark}
\def\theremark {\arabic{remark}}
\newenvironment{remark}{\refstepcounter{remark}\par\noindent{\bf Remark\ \theremark}\ }{\par}
\newtheorem{Proof}{Proof}

\newenvironment{proof}{\begin{Proof}\rm}{\hfill $\Box$ \end{Proof}}

\usepackage{hyperref}

\title{POD-ROM methods: from a finite
set of snapshots to continuous-in-time approximations}
\author{Bosco
Garc\'{\i}a-Archilla\thanks{Departamento de Matem\'atica Aplicada
II, Universidad de Sevilla, Sevilla, Spain. Research is supported by
Spanish MCINYU under grants PGC2018-096265-B-I00 and PID2019-104141GB-I00 (bosco@esi.us.es)}
\and  Volker John\thanks{Weierstrass Institute for Applied Analysis and Stochastics,
Leibniz Institute in Forschungsverbund Berlin e. V. (WIAS), Mohrenstr. 39, 10117 Berlin, Germany.
Freie Universit\"at of Berlin,
Department of Mathematics and Computer Science,
Arnimallee 6, 14195 Berlin, Germany, ({john@wias-berlin.de}).}
  \and Julia Novo\thanks{Departamento de
Matem\'aticas, Universidad Aut\'onoma de Madrid, Spain. Research is supported
by Spanish MINECO
under grants PID2019-104141GB-I00 and VA169P20  (julia.novo@uam.es)}}
\date{\today}
\begin{document}
\maketitle

\abstract{This paper studies discretization of time-dependent partial differential equations (PDEs) by proper orthogonal decomposition 
reduced order models (POD-ROMs). Most of the analysis in the literature has been performed on fully-discrete methods using
first order methods in time, typically the implicit Euler time integrator. 
Our aim is to show which kind of error bounds can be obtained using any time integrator, both in the full order model (FOM), applied to
compute the snapshots, and in the POD-ROM method.  To this end, we analyze in this paper the continuous-in-time case for both
the FOM and POD-ROM methods, although the POD basis is obtained from snapshots taken at a discrete (i.e., not continuous) set times.
Two cases for the set of snapshots are considered: The case in which the
snapshots are based on first order divided differences in time and the case in which they are based on temporal derivatives. Optimal pointwise-in-time error bounds {between the FOM and the POD-ROM solutions} are proved for the $L^2(\Omega)$ norm of the error for a semilinear reaction-diffusion model problem. The dependency of the errors on the distance in time between two consecutive snapshots 
and on the tail of the POD eigenvalues is tracked. 
Our detailed analysis allows to show that, in some situations, a small number of snapshots in a given time interval might be sufficient to accurately approximate the solution in the full interval.  
Numerical studies support the error analysis.}
\bigskip


\section{Introduction}
Although reduced order models (ROMs) use a reduced (tailored) basis obtained from a proper orthogonal decomposition (POD) of a (finite) set of approximations (snapshots) at a discrete set of times $0=t_0<t_1< \ldots <t_M=T$ to approximate a time-dependent problem, in the present paper we derive error bounds of POD-ROM approximations for all~$t\in[0,T]$ and not only for the discrete set of snapshot times. Since POD-ROM methods are based on the same variational formulation as standard finite element methods (FEMs), we consider the continuous-in-time case to study approximations for all~$t\in[0,T]$, or, in other words, the semidiscretization in space of the POD-ROM method's time-dependent problem, which is the POD-ROM version of the FEM semidiscretization in space of  the same time-dependent problem, also known in the literature as full-order model (FOM).

Most of the references in the literature on POD-ROM methods analyze the fully-discrete case in which the implicit Euler method is used as time integrator in the FOM and the corresponding POD-ROM. 
One important question is the possibility of using different 
time integrators for the FOM and for the POD-ROM. This question arises naturally when one is given a set of snapshots without details on how they were obtained.
Another question is the possibility of using different time steps for the FOM and for the POD-ROM.  This arises naturally when the approximation is wanted at different times than those of the snapshots, for example, when one applies  a time integrator with a sufficiently small time step to have  negligible temporal errors (as part of the offline computations) and keep only some of the snapshots, 
which contain in some sense the most essential information, to 
generate the reduced basis. To cover all possible scenarios, in this paper we analyze a continuous-in-time POD method, since although in practice one uses a time integrator, this obtains approximations to the solution of the continuous-in-time problem. Also, for simplicity, we assume that the snapshots are exact solutions of the continuous-in-time FEM approximation, which is a simplification of the case where they have been obtained in practice with high accuracy using small time steps.

For most of the POD-ROM methods studied in the literature, the set of snapshots is based on full order approximations at different time instants. Some approaches include also first order divided differences, as it is the case of the pioneering paper \cite{Ku-Vol}.
It is shown in \cite{samu_et_al} for the heat equation that adding first order divided differences allows to derive pointwise-in-time error estimates. Later in \cite{lo-sin}, it is proved,  also for the heat equation,
that only snapshots based on first order divided differences plus one snapshot for the initial time are needed. In this way the cardinality of the set of snapshots is reduced by half. 
In \cite{nos_der_pod} a set of snapshots based on Galerkin time derivatives, instead of finite differences, is considered
in a POD-ROM method applied to the Navier--Stokes equations. As in \cite{lo-sin}, apart from the initial time, that can be
replaced by the mean value of the snapshots, no other snapshots at different times are necessary. It is proved in 
\cite{nos_der_pod} that the use of time derivatives also allows to obtain pointwise-in-time error estimates.

In the present paper we consider as a model problem an evolutionary diffusion-reaction equation with a nonlinear reactive 
term. 
We study both, the case in which first order divided differences and the case in which Galerkin time derivatives form the set of snapshots. POD projections with respect to the $H_0^1(\Omega)$ norm are considered. Using 
$H_0^1(\Omega)$ orthogonal projections is in some sense the natural approach for parabolic evolutionary problems and gives (theoretically) better bounds in terms of the tail of the eigenvalues than 
$L^2(\Omega)$ projections, see \cite{samu_et_al}. We study how the errors depend on the distance in time between two consecutive snapshots, $\Delta t$,  and
on the tail of the eigenvalues computed with the POD.
For both sets of snapshots optimal pointwise-in-time estimates (for the $L^2(\Omega)$ error in space) are derived. 
It is proved that the error of the POD-ROM behaves as $(\Delta t)^q$ where the value of $q$ depends on the regularity of the continuous-in-time Galerkin approximation, more precisely,
on the size of the $q+1$ temporal derivative. As a consequence, in some cases, a set of snapshots based on a very small number of time instants in a given interval is enough to accurately approximate the solution in the full interval.
Note that obtaining optimal  pointwise-in-time error bounds for POD-ROM methods that do not include  first order divided differences
or time derivatives seems not to be possible with the currently known analytic approaches, 
see \cite{Ku-Vol,samu_et_al,nos_der_pod,lo-sin}.

Second order bounds in time are proved recently in \cite{ap_num_let_bdf2} for a POD-ROM applied to the heat equation 
with snapshots based on first order divided differences and BDF2 as time integrator in the POD-ROM method.
For the present paper, we have decided to focus on the continuous-in-time POD-ROM method. Although the case in 
which a fully-discrete method is applied in the POD-ROM can be analyzed following \cite{nos_der_pod,ap_num_let_bdf2}, the technical details would increase the length of the paper considerably. For this reason, the fully-discrete case is
outside the scope of the present paper. %subject of a different paper. 

The outline of the paper is as follows. In the first section, some preliminaries and notation are introduced. 
Section~\ref{sec:POD} describes the sets of snapshots and their POD decomposition. 
In Section \ref{section4} we prove some preliminary results that will be used in the error analysis of the method in 
Section~\ref{section5}. For the technical lemmas of Section~\ref{section4}  we have decided to prove in detail the simplest case and then indicate how to extend the results to the general case. We think that in this way the ideas of the proofs can be understood more easily. Some numerical results are shown in Section~\ref{sec:numres} for a problem with polynomial
nonlinearity. The numerical studies support the analytical results and discuss how to choose snapshots in the considered
example in a good way.

\section{Preliminaries and notation}
As a model problem we consider the following reaction-diffusion equation 
\begin{equation}\label{eq:model}
\begin{array}{rcll}
u_t(t,x)-\nu\Delta u(t,x)+g(u(t,x))& = &f(t,x) &\quad (t,x)\in (0,T]\times \Omega,\\
u(t,x)&=&0,& \quad  (t,x)\in (0,T]\times \partial \Omega,\\
u(0,x)&=&u^0(x),&\quad  x\in \Omega,
\end{array}
\end{equation}
in a bounded domain $\Omega\subset {\Bbb R}^d$, $d\in \{2,3\}$, where $g$ is a nonlinear smooth function.

We will use standard notation for Sobolev spaces and norms. 
In the sequel we will assume for the solution of \eqref{eq:model}
\begin{equation}\label{eq:bounds_u}
\|u(t)\|_\infty\le C, \quad \forall\ t\in[0,T].
%,\quad \|u_t(t)\|_\infty\le C,\quad \forall\ t\in[0,T].
\end{equation}
We will denote by $C_p$ the constant in the Poincar\'e inequality
\begin{equation}\label{poincare}
\|v\|_0\le C_p\|\nabla v\|_0,\quad v\in H_0^1(\Omega).
\end{equation}

Let us denote by $X_h^l$ a finite element space based on piece-wise continuous polynomials of degree $l$  defined
over a triangulation of $\Omega$ of diameter $h$ 
and by $V_h^l$ the  finite element space based on piece-wise continuous polynomials of degree $l$ that satisfies also the homogeneous Dirichlet boundary conditions at the boundary $\partial \Omega$.

The following inverse inequality holds for all $v_h\in X_h^l$ , see
\cite[Theorem 3.2.6]{Cia78},
\begin{equation}
\label{inv} \| v_{h} \|_{W^{m,p}(K)} \leq c_{\mathrm{inv}}
h_K^{n-m-d\left(\frac{1}{q}-\frac{1}{p}\right)}
\|v_{h}\|_{W^{n,q}(K)},
\end{equation}
$0\le n\le m\le 1$, $1\le q\le p\le \infty$.

In the sequel,  $I_h u \in X_h^l$ will denote
the Lagrange interpolant of a continuous function $u$. The following bound can be found in \cite[Theorem 4.4.4]{brenner-scot}
\begin{equation}\label{cota_inter}
|u-I_h u|_{W^{m,p}(K)}\le c_\text{\rm int} h^{n-m}|\bu|_{W^{n,p}(K)},\quad 0\le m\le n\le l+1,
\end{equation}
where $n>d/p$ when $1< p\le \infty$ and $n\ge d$ when $p=1$.

An inverse inequality for the $L^\infty(\Omega)$ norm of a finite element function is shown next. 

\begin{lema} Let $v_h\in V_h^l$, then it holds
\begin{equation}\label{inv_b}
\|v_h\|_\infty\le C h^{-1/2}\|\nabla v_h\|_0.
\end{equation}
\end{lema}
\begin{proof}
 We first observe that \cite[(4.39)]{heyran1}
\begin{equation}\label{heyran}
\|v_h\|_\infty\le C\|\nabla v_h\|_0^{1/2}\|A_h v_h\|_0^{1/2},
\end{equation}
where $A_h:V_h^l\rightarrow V_h^l$ is defined as
\begin{equation}\label{defi}
(\nabla v_h,\nabla w_h) = (A_h v_h,w_h) = (v_h,A_h w_h),\quad  \forall\ v_h,w_h\in V_h^l.
\end{equation}
Now, taking $w_h=A_h v_h $ in \eqref{defi} and applying \eqref{inv}
\begin{eqnarray*}
\|A_h v_h\|_0^2=(\nabla v_h,\nabla A_h v_h)\le \|\nabla v_h\|_0\|\nabla A_h v_h\|_0\le c_{\mathrm{inv}} h^{-1}\|\nabla v_h\|_0\| A_h v_h\|_0,
\end{eqnarray*}
we get
$$
\|A_h v_h\|_0\le c_{\mathrm{inv}} h^{-1}\|\nabla v_h\|_0.
$$
Inserting this inequality in \eqref{heyran}, we reach \eqref{inv_b}.
\end{proof}

For a function $v\in H^1(I)$, $I=(a,b)$, $a,b\in {\Bbb R}$, $v(a)=v(b)=0$, Agmon's inequality states 
\begin{equation}\label{agmon}
\|v\|_\infty\le \|v\|_0^{1/2}|v|_1^{1/2}.
\end{equation}
This inequality can be derived from  
\begin{eqnarray*}
v^2(x) &=&2\int_a^x  v(s)v'(s)\ \di s\le 2\int_a^x |v(s)||v'(s)|\ \di s,\\
v^2(x) &=&-2\int_x^b  v(s)v'(s)\ \di s\le 2\int_x^b |v(s)||v'(s)|\ \di s,
\end{eqnarray*}
adding both inequalities, applying the Cauchy--Schwarz inequality, and taking the square root.

The semi-discrete Galerkin approximation reads as follows: Find $u_h\ :\ (0,T]\rightarrow V_h^l$ such that
\begin{eqnarray}\label{gal_semi}
(u_{h,t},v_h)+\nu(\nabla u_h,\nabla v_h)+(g(u_h),v_h)=(f,v_h),\quad \forall\ v_h\in V_h^l,
\end{eqnarray}
with $u_h(0)=I_h u^0\in V_h^l.$ 
If the weak solution of \eqref{eq:model} is sufficiently smooth, then the following error estimate is well known,
see for example \cite{bosco-titi-fem},
\begin{equation}\label{cota_gal}
\max_{0\le t\le T}\left(\|(u-u_h)(t)\|_0+h\|(u-u_h)(t)\|_1\right)\le C(u)h^{l+1}.
\end{equation}

Let us denote by $s_h$ the elliptic projection of $u$ that is define as $s_h\ :\ [0,T]\rightarrow V_h^l$
such that
\[
(\nabla s_h(t), \nabla v_h)=(\nabla u(t),\nabla v_h),\quad \forall\ v_h\in V_h^l.
\]
The following error estimation is well known, \cite{brenner-scot}
\begin{equation}\label{cota_sh}
\max_{0\le s\le T}\left(\|(u-s_h)(s)\|_0+h\|(u-s_h)(s)\|_1\right)\le C(u)h^{l+1}.
\end{equation}
Moreover, taking into account that $\|I_h u\|_\infty\le C\|u\|_\infty$ and applying inverse inequality \eqref{inv} together with  \eqref{cota_sh}, \eqref{cota_inter} 
and \eqref{eq:bounds_u} we get for $u(t) \in H^2(\Omega)$
\begin{eqnarray}\label{cota_shinf}
\|s_h(t)\|_\infty&\le& \|s_h(t)-I_h u(t)\|_\infty+\|I_h u(t)\|_\infty\le c_{\mathrm{inv}} h^{-d/2}\|s_h(t)-I_h u(t)\|_0+C\|u(t)\|_\infty\nonumber\\
&\le& C h^{-d/2}h^2\|u(t)\|_2+C\|u(t)\|_\infty\le C, \quad t\in[0,T].
\end{eqnarray}
\iffalse
To derive the error bound for the $H^1(\Omega)$ norm stated in \eqref{cota_gal}, one uses 
\eqref{cota_sh}, \eqref{inv}, and the bound for the $L^2(\Omega)$ error to find
\begin{eqnarray}\label{eq:bound_H1}
\|\nabla(u-u_h)(t)\|_0 &\le& \|\nabla(u-s_h)(t)\|_0 + \|\nabla(s_h-u_h)(t)\|_0 \nonumber\\
&\le&  C(u)h^{l} +  c_{\mathrm{inv}}h^{-1}  \|(s_h-u_h)(t)\|_0 \\
&\le& 
 C(u)h^{l} +  c_{\mathrm{inv}}h^{-1}  \|(u-u_h)(t)\|_0 + c_{\mathrm{inv}}h^{-1}  \|(u-s_h)(t)\|_0
 \le C(u)h^{l}. \nonumber
\end{eqnarray}
\fi

Next, it is shown that the finite element solution belongs to $L^\infty(\Omega)$ for each time. 

\begin{lema}  Let the weak solution of \eqref{eq:model} be sufficiently smooth, then  the following bound holds
\begin{eqnarray}\label{uhuht_inf}
\|u_h(t)\|_\infty\le c_{\mathrm{inf}}, \quad t\in[0,T].
%,\quad \|u_{h,t}(t)\|_\infty\le c_{t,inf}.
\end{eqnarray}
\end{lema}
\begin{proof}
To prove \eqref{uhuht_inf} we add and subtract $s_h$ and apply \eqref{inv_b}, followed by \eqref{cota_gal}, \eqref{cota_sh} and \eqref{cota_shinf}
\[
\|u_h(t)\|_\infty\le \|(u_h-s_h)(t)\|_\infty+\|s_h(t)\|_\infty\le C h^{-1/2}\|\nabla (u_h-s_h)(t)\|_0+\|s_h(t)\|_\infty.
\]
Utilizing now \eqref{cota_gal}, \eqref{cota_sh} and \eqref{cota_shinf} and taking into account $l\ge 1$ we conclude
\[
\|u_h(t)\|_\infty\le C h^{-1/2}h^l+C\le C.
\]
\end{proof}

Performing only a discretization in space, then $u_h$ is the solution of an initial value problem associated to a system of ordinary differential equations. Thus, the temporal regularity of $u_h$ depends essentially on the temporal regularity of the 
nonlinear term $g$ and the right-hand side of the equation $f$. It will be assumed throughout this paper
that the temporal regularity of these functions is sufficiently high so that all temporal derivatives
of $u_h$ and the corresponding norms that appear in the analysis are well defined.

\section{Proper orthogonal decomposition}\label{sec:POD}

This section describes the two studied approaches for computing a basis for the POD-ROM simulations. 

\subsection{Finite differences with respect to time case}\label{sec:pod_fdm}

Let us fix $T>0$ and take $\Delta t=T/M$. Let $t_j=j\Delta t$, $j=0,\ldots,M$ and let $N=M+1$. We define the 
space
$$
\bU = {\rm span}\left\{\sqrt{N}w_0,\tau\frac{u_h(t_1)-u_h(t_0)}{\Delta t},\tau\frac{u_h(t_2)-u_h(t_1)}{\Delta t},\ldots,\tau\frac{u_h(t_M)-u_h(t_{M-1})}{\Delta t}\right\},\\
$$
where $w_0$ is either $w_0=u_h(t_0)$ or $w_0=\overline u_h=\sum_{j=0}^Mu_h(t_j)/(M+1)$,
and $\tau$ is a time scale to make the snapshots dimensionally correct. {The
following analysis only requires $\tau>0$.} 
We denote
$$
y_h^1=\sqrt{N}w_0, \qquad y_h^j = \tau\frac{u_h(t_{j-1})-u_h(t_{j-2})}{\Delta t}, \quad j=2,\ldots,M+1=N,
$$
so that $\bU={\rm span}\left\{y_h^1,y_h^2,\ldots,y_h^N\right\}$.

Let 
$X=H_0^1(\Omega)$ and let us denote the correlation matrix by $K=((k_{i,j}))\in {\Bbb R}^{N\times N}$ with
$$
k_{i,j}=\frac{1}{N}(y_h^i,y_h^j)_X=\frac{1}{N}(\nabla y_h^i,\nabla y_h^j),\quad i,j=1,\ldots,N,
$$
and $(\cdot,\cdot)$ the inner product in $L^2(\Omega)$. We denote by $\lambda_1\ge \lambda_2\ldots\ge \lambda_{d_r}>0$ the positive eigenvalues of $K$ and
by $\bv_1,\ldots,\bv_{d_r}\in {\Bbb R}^N$ the associated eigenvectors. The orthonormal POD basis functions of $\bU$ are
\begin{equation}\label{eq:varphi}
\varphi_k=\frac{1}{\sqrt{N}}\frac{1}{\sqrt{\lambda_k}}\sum_{j=1}^Nv_k^j y_h^j,
\end{equation}
where $v_k^j$ is the $j$ component of the eigenvector $\bv_k$.
For any $1\le r\le d_r$ let 
\begin{equation}\label{eq:bU_r}
\bU^r={\rm span}\left\{\varphi_1,\varphi_2,\ldots,\varphi_r\right\},
\end{equation}
and let us denote by $P^r:H_0^1(\Omega)\rightarrow \bU^r$ the $H_0^1$-orthogonal projection onto $\bU^r$. Then, it holds
\begin{equation}\label{cota_pod}
\frac{1}{N}\sum_{j=1}^N\|\nabla(y_h^j-P^r y_h^j)\|_0^2=\sum_{k={r+1}}^{d_r}\lambda_k.
\end{equation}

\subsection{Time derivatives case} 
Using the same notations as in Section~\ref{sec:pod_fdm}, we define 
\[
\bU = {\rm span}\left\{\sqrt{N}w_0,\tau u_{h,t}(t_0),\tau u_{h,t}(t_1),\ldots,\tau u_{h,t}(t_M)\right\},
\]
where now $N=M+2$,
$$
w_0=u_h(t_0)\quad {\rm or}\quad w_0=\overline u_h=\frac{\sum_{j=0}^Mu_h(t_j)}{M+1},
$$
and $\tau$ is a time scale that makes the snapshots dimensionally correct.
Let us denote by
$$
y_h^1=\sqrt{N} w_0,\qquad y_h^j=u_{h,t}(t_{j-2}),\quad j=2,\ldots,M+2=N,
$$
so that
$\bU={\rm span}\left\{y_h^1,y_h^2,\ldots,y_h^N\right\}.$

Analogous expressions as \eqref{eq:varphi}-\eqref{eq:bU_r} for the POD basis functions and the POD space can be written, taking into account that the value of 
$N$ is different than for the finite difference  case.

Denoting by $P^r:H_0^1(\Omega)\rightarrow \bU^r$ the $H_0^1$-orthogonal projection onto $\bU^r$ it holds
\begin{equation}\label{cota_pod_time_deri}
\frac{1}{N}\sum_{j=1}^N\|\nabla (y_h^j-P^r y_h^j)\|_0^2=\sum_{k={r+1}}^{d_r}\lambda_k.
\end{equation} 

\iffalse
With the same notation as before, 

the orthonormal POD basis function of $\bU$ are
$$
\varphi_k=\frac{1}{\sqrt{N}}\frac{1}{\sqrt{\lambda_k}}\sum_{j=1}^Nv_k^k y_h^j,
$$
where $v_k^j$ is the $j$ component of the eigenvector $\bv_k$.
As before, for any $1\le r\le d_r$ we denote by
$$
\bU^r={\rm span}\left\{\varphi_1,\varphi_2,\ldots,\varphi_r\right\},
$$
and by $P^r:H_0^1(\Omega)\rightarrow \bU^r$ the $X$-orthogonal projection onto $\bU^r$. Then, it holds
\begin{equation}\label{cota_pod_time_deri}
\frac{1}{N}\sum_{j=1}^N\|\nabla (y_h^j-P^r y_h^j)\|_0^2=\sum_{k={r+1}}^{d_r}\lambda_k.
\end{equation} 
\fi

\section{Preliminary results}\label{section4}

\begin{lema}\label{le:bosco}
Let $\varphi :[0,T]\times \Omega\rightarrow {\Bbb R}$ be a regular enough function with  $\frac{\partial^2 \varphi }{\partial t^2}\in L^2(t_0,T,H^1)$. Then, 
the following bound holds
\begin{eqnarray}\label{new_cota_int}
\int_{t_0}^T \|(I-P^r)\varphi (t)\|_{0}^2 \ \di t&\le& \frac{8}{3}\sum_{n=0}^M (\Delta t)\|(I-P^r) \varphi (t_n)\|_0^2\nonumber\\
&\quad&+2C_p^2c_{\rm int}^2 (\Delta t)^4
\int_{t_0}^{T}\left\| \frac{\partial^2 \nabla \varphi (t)}{\partial t^2}\right\|_0^2\ \di t.
\end{eqnarray}
\end{lema}
\begin{proof}
For any $t\in[t_n,t_{n+1}]$ we split
$$
\varphi (t)=I_{2,t} \varphi (t)+R_{2,t} \varphi (t),
$$
where $I_{2,t} \varphi $ is the Lagrange linear interpolant in time defined as
$$
I_{2,t} \varphi (t)=\varphi (t_n)\frac{(t_{n+1}-t)}{\Delta t}+\varphi (t_{n+1})\frac{(t-t_n)}{\Delta t},
$$
and $R_{2,t} \varphi $ is the error
$$
%R_{2,t} \varphi (s)=\frac{(s-t_k)(s-t_{k+1})}{2}\varphi _{tt}(\xi_s),\quad \xi_s\in[t_{k},t_{k+1}].
R_{2,t} \varphi (t)=\varphi (t)-I_{2,t} \varphi (t).
$$
Applying Poincar\'e inequality \eqref{poincare} and taking into account that for any function $v\in H_0^1(\Omega)$, $\|\nabla(I-P^r)  v\|_0\le \|\nabla v\|_0$ we get
\begin{eqnarray*}
\|(I-P^r)\varphi (t)\|_0^2&\le& 2 \|(I-P^r) I_{2,t} \varphi (t)\|_0^2+ 2\|(I-P^r) R_{2,t} \varphi (t)\|_0^2\nonumber\\
&\le& 2 \|(I-P^r) I_{2,t} \varphi (t)\|_0^2+ 2C_p^2\|\nabla (I-P^r) R_{2,t} \varphi (t)\|_0^2\nonumber\\
&\le& 2 \|(I-P^r) I_{2,t} \varphi (t)\|_0^2+ 2C_p^2\|\nabla R_{2,t} \varphi (t)\|_0^2\nonumber\\
&=&2 \|(I-P^r) I_{2,t} \varphi (t)\|_0^2+ 2C_p^2\| R_{2,t} \nabla \varphi (t)\|_0^2.
\end{eqnarray*}
It follows that 
\begin{eqnarray}\label{eq:bosco1}
\int_{t_n}^{t_{n+1}}\|(I-P^r)\varphi (t)\|_0^2 \ \di t &\le& 2 \int_{t_n}^{t_{n+1}} \|(I-P^r) I_{2,t} \varphi (t)\|_0^2\ \di t\nonumber\\
&&+ 2C_p^2\int_{t_n}^{t_{n+1}}\| R_{2,t} \nabla \varphi (t)\|_0^2\ \di t.
\end{eqnarray}
To bound the first term on the right-hand side of \eqref{eq:bosco1} we observe that
\begin{eqnarray*}
\|(I-P^r) I_{2,t} \varphi (t)\|_0^2&\le& 2\int_\Omega |(I-P^r)\varphi (t_n)|^2\left(\frac{(t_{n+1}-t)}{\Delta t}\right)^2\ \di x
\nonumber\\
&&+2 \int_\Omega |(I-P^r)\varphi (t_{n+1})|^2\left(\frac{(t-t_n)}{\Delta t}\right)^2\ \di x,
\end{eqnarray*}
so that 
\begin{eqnarray}\label{eq:bosco2}
\int_{t_n}^{t_{n+1}}\|(I-P^r)I_{2,t}\varphi (t)\|_0^2\ \di t \le \frac{2}{3}\Delta t \|(I-P^r)\varphi (t_n)\|_0^2+ \frac{2}{3} \Delta t \|(I-P^r)\varphi (t_{n+1})\|_0^2.
\end{eqnarray}
For the second term on the right-hand side of \eqref{eq:bosco1}  we can apply the general bound \eqref{cota_inter} for the
Lagrange interpolant, in time instead of in space, to get 
\begin{eqnarray}\label{eq:bosco3}
\lefteqn{\int_{t_n}^{t_{n+1}}\| R_{2,t} \nabla \varphi (t)\|_0^2\ \di t}\nonumber\\
&=&\int_{t_n}^{t_{n+1}}\int_\Omega | R_{2,t}\nabla \varphi (t,x) |^2\ \di x \di t= \int_\Omega \int_{t_n}^{t_{n+1}}| (I-I_{2,t})\nabla \varphi (t,x)|^2\ \di t \di x\nonumber\\
&\le& c_\text{\rm int}^2 (\Delta t)^4\int_\Omega \int_{t_n}^{t_{n+1}} \left| \frac{\partial^2 \nabla \varphi (t,x)}{\partial t^2}\right|^2\ \di t \di x= c_\text{\rm int}^2 (\Delta t)^4\int_{t_n}^{t_{n+1}}  \left\| \frac{\partial^2 \nabla \varphi (t)}{\partial t^2}\right\|_0^2\ \di t.
\end{eqnarray}
Inserting \eqref{eq:bosco2} and \eqref{eq:bosco3} into \eqref{eq:bosco1} yields
\begin{eqnarray}\label{eq:bosco4}
\int_{t_n}^{t_{n+1}}\|(I-P^r)\varphi (t)\|_0^2\ \di t   &\le& \frac{4}{3}\Delta  t\|(I-P^r)\varphi (t_n)\|_0^2+\frac{4}{3}\Delta t\|(I-P^r)\varphi (t_{n+1})\|_0^2\nonumber\\
&&+2C_p^2 c_\text{\rm int}^2 (\Delta t)^4\int_{t_n}^{t_{n+1}}  \left\| \frac{\partial^2 \nabla \varphi (t)}{\partial t^2}\right\|_0^2\ \di t.
\end{eqnarray}
Writing
$$
\int_{t_0}^T \|(I-P^r)\varphi (t)\|_0^2 \ \di t = \sum_{n=0}^{M-1}\int_{t_n}^{t_{n+1}}\|(I-P^r)\varphi (t)\|_0^2\ \di t
$$
and applying \eqref{eq:bosco4} we reach \eqref{new_cota_int}.
\end{proof}

\begin{remark}\label{rem:rem1}
Assuming enough regularity for $\varphi $ and using higher order Lagrange interpolants in time instead of $I_{2,t}$ in the
proof of Lemma \ref{le:bosco}, one can get analogous results increasing the power of $\Delta t$. More precisely, one can obtain
the following generalization of \eqref{new_cota_int} for $q\ge 2$
\begin{equation}\label{new_cota_int_orp}
\int_{t_0}^T \|(I-P^r)\varphi (t)\|_0^2 \ \di t\le C\sum_{n=0}^M (\Delta t)\|(I-P^r) \varphi (t_n)\|_0^2+CC_p^2 (\Delta t)^{2q}
\int_{t_0}^{T}\left\| \frac{\partial^q \nabla \varphi (t)}{\partial t^q}\right\|_0^2 \ \di t,
\end{equation}
for a constant that depends only on~$q$ and~$c_{\rm int}$.
\end{remark}

\subsection{Finite differences with respect to time case}

The proof of the following lemma is derived with the same arguments as used in the proof of~\cite[Lemma 2]{ap_num_let_bdf2}  (see also \cite{samu_et_al}).
\begin{lema}\label{le:maxPr_FD} Let $\tilde C=1$ if $w_0=u_h(t_0)$ and $\tilde C=4$ if~$w_0=\overline u_h$.
The following bound holds
\begin{eqnarray}\label{max_dif}
\max_{0\le n\le M}\|u_h^n-P^ru_h^n\|_0^2&\le& \left(2+4\tilde C \frac{T^2}{\tau^2}\right)C_p^2\sum_{k={r+1}}^{d_r}\lambda_k,\nonumber\\
\max_{0\le n\le M}\|\nabla(u_h^n-P^ru_h^n)\|_0^2&\le& \left(2+4\tilde C \frac{T^2}{\tau^2}\right)\sum_{k={r+1}}^{d_r}\lambda_k.
\end{eqnarray}
As a consequence, since $\Delta t =T/M$ and  $(M+1)/M\le 2$, it follows that
\begin{equation}\label{max_dif_promedio}
{\Delta t }\sum_{n=0}^M \|u_h^n-P^ru_h^n\|_0^2\le T\left(4+8\tilde C \frac{T^2}{\tau^2}\right)C_p^2\sum_{k={r+1}}^{d_r}\lambda_k.
\end{equation}
\end{lema}

\begin{lema}\label{le:tra_cota_pro_b}
For each $q\ge 2$, there exist a constant~$C$ such that for $1\le r\le {d_r}$ the following bound holds

 \begin{equation}\label{eq:tra:pro_b}
\int_{t_0}^{T} \|(I-P^r)u_h(t)\|_0^2\ \di t\le CC_p^2 \left(T\left(1+ \frac{T^2}{\tau^2}\right)\sum_{k={r+1}}^{d_r}\lambda_k 
+(\Delta t)^{2q}
\int_{t_0}^{T}\left\| \frac{\partial^{q} \nabla u_h(t)}{\partial t^q}\right\|_0^2 \ \di t\right),
\end{equation}
under the assumption that $u_h$ is smooth enough such that the last 
term in \eqref{eq:tra:pro_b} is well defined. 

\end{lema}
\begin{proof}
We apply \eqref{new_cota_int_orp} to $\varphi(t)=u_h(t)$ and use \eqref{max_dif_promedio} for bounding the 
first term on the right-hand side of \eqref{new_cota_int_orp}.
\end{proof}

\begin{lema}\label{le:deri_by_dif}
Let $g\ :\ [0,T]\times \Omega\rightarrow {\Bbb R}$ with $\frac{\partial^3 g}{\partial t^3}\in L^2(0,T;L^2)$,
then there exists a constant $C>0$ such that the following bound holds
\begin{eqnarray}\label{deri_dif}
\sum_{n=0}^M(\Delta t )\|g_t^n\|_0^2\le C \sum_{n=1}^M (\Delta t)\|D g^n\|_0^2+C (\Delta t )^4 \int_{t_0}^T\left\|\frac{\partial^{3} g(t)}{\partial t^{3}}\right\|_0^2\ \di t,
\end{eqnarray}
where $g_t^n=g_t(t_n)$ and $D g^n=(g^n-g^{n-1})/\Delta t$.
\end{lema}
\begin{proof}
Let $n\ge 2$ and let us denote by $D^2g^n=((3/2)g^n-2 g^{n-1}+(1/2)g^{n-2})/(\Delta t)$, the value at $t_n$ of the derivative of the quadratic Lagrange interpolant on the nodes $t_{n-2}$, $t_{n-1}$, and~$t_n$. Using Taylor series expansion
with integral remainder, we obtain
\begin{eqnarray*}
g_t^n&=&\frac{3}{2}D g^n-\frac{1}{2}Dg^{n-1}+(g_t^n-D^2 g^n)\\
&=&\frac{3}{2}D g^n-\frac{1}{2}Dg^{n-1}+
\frac{1}{2\Delta t} \int_{t_{n-2}}^{t_n}\left(2(t-t_{n-1})_{+}^2 -\frac{1}{2}(t-t_{n-2})^2\right)\frac{\partial^3 g(t)}{\partial t^3}\ \di t,
\end{eqnarray*}
where $x_{+}=\max(0,x)$,  for $x\in {\Bbb R}$. Applying the Cauchy--Schwarz inequality followed by a straightforward calculation gives
\begin{eqnarray*}
|g_t^n|&\le&\frac{3}{2}\left(|D g^n|+|Dg^{n-1}|\right)+
\left(\frac{1}{\sqrt{5}} + \frac{\sqrt{2}}{\sqrt{5}}\right) (\Delta t)^{3/2} \left(\int_{t_{n-2}}^{t_n} \left(\frac{\partial^3 g(t)}{\partial t^3}\right)^2\ \di t\right)^{1/2}\nonumber\\&\le&
\frac{3}{2}\left(|D g^n|+|Dg^{n-1}|\right)+\frac{5}{2\sqrt{5}} (\Delta t)^{3/2} \left(\int_{t_{n-2}}^{t_n}\left(\frac{\partial^3 g(t)}{\partial t^3}\right)^2\ \di t\right)^{1/2},\ n\ge 2.
\end{eqnarray*}
Hence, it is 
\[
|g_t^n|^2\le C \left(|D g^n|^2+|Dg^{n-1}|^2\right)+C (\Delta t)^3\int_{t_{n-2}}^{t_n} \left(\frac{\partial^3 g(t)}{\partial t^3}\right)^2\ \di t,
\]
so that for $2\le n\le M$
\begin{eqnarray}\label{deri_dif2}
\|g_t^n\|_0^2\le C\left(\|D g^n\|_0^2+\|Dg^{n-1}\|_0^2\right)+C (\Delta t)^3\int_{t_{n-2}}^{t_n} \left\|\frac{\partial^{3} g(t)}{\partial t^{3}}\right\|_0^2\ \di t.
\end{eqnarray}

For $n=1$ 
we can approximate
$g_t^1$ by $\frac{1}{2}D g^1+\frac{1}{2}D g^2$ and for $n=0$ we can approximate $g_t^0$ by
$\frac{3}{2}D g^1-\frac{1}{2}D g^2$, both approximations being of second order consistency. Arguing similarly as above, yields for $n=0,1$
\begin{equation}\label{deri_dif3}
\|g_t^n\|_0^2\le C\left(\|D g^1\|_0^2+\|Dg^{2}\|_0^2\right)+C (\Delta t)^3\int_{t_{0}}^{t_2} \left\|\frac{\partial^{3} g(t)}{\partial t^{3}}\right\|_0^2\ \di t.
\end{equation}

From \eqref{deri_dif2} and \eqref{deri_dif3} it follows that 
\begin{eqnarray*}
\sum_{n=0}^M(\Delta t )\|g_t^n\|_0^2\le C \sum_{n=1}^M (\Delta t)\|D g^n\|_0^2+C (\Delta t)^4\sum_{n=2}^M\int_{t_{n-2}}^{t_n} \left\|\frac{\partial^{3} g(t)}{\partial t^{3}}\right\|_0^2\ \di t
\end{eqnarray*}
which gives \eqref{deri_dif}.
\end{proof}

\begin{remark}
Since we can approximate a time derivative at $t_n$ with order $q\ge 2$ using a linear combination 
 of first order divided differences, the arguments of the proof of Lemma~\ref{le:deri_by_dif} can be extended to obtain
for $q\ge 2$ and functions $g$ with $\frac{\partial^{q+1}g}{\partial t^{q+1}}\in L^2(0,T;L^2)$
\begin{eqnarray}\label{deri_dif_p}
\sum_{n=0}^M(\Delta t )\|g_t^n\|_0^2\le C \sum_{n=1}^M (\Delta t)\|D g^n\|_0^2+C(\Delta t )^{2q} \int_{t_0}^T\left\|\frac{\partial^{q+1} g(t)}{\partial t^{q+1}}\right\|_0^2 \ \di t,
\end{eqnarray}
for a constant~$C$ depending only on~$q$.
\end{remark}

\begin{lema}\label{le:tra_cota_pro_t_b} For each $q\ge 2$ there exist a constant $C>0$ such that 
 for $1\le r\le d_r$ the following estimate holds:
 \begin{equation}\label{eq:tra:pro_t_b}
\int_{t_0}^{T} \|(I-P^r)u_{h,t}(t)\|_0^2\ \di t\le CC_p^2\left(\frac{T}{\tau^2}\sum_{k=r+1}^{d_r} \lambda_k+(\Delta t)^{2q}
\int_{t_0}^{T}\left\| \frac{\partial^{q+1} \nabla u_h(t)}{\partial t^{q+1}}\right\|_0^2 \ \di t\right),
\end{equation}
under the assumption that $u_h$ is smooth enough such that the last 
term in \eqref{eq:tra:pro_t_b} is well defined. 
\end{lema}

\begin{proof}
We apply \eqref{new_cota_int_orp}  to $\varphi(s)=u_{h,t}(s)$ so that
\begin{eqnarray}\label{mecu1}
\int_{t_0}^{T} \|(I-P^r)u_{h,t}(t)\|_0^2\ \di t&\le& C\sum_{n=0}^M (\Delta t)\|(I-P^r) u_{h,t}^n\|_0^2\nonumber\\
&&+CC_p^2 (\Delta t)^{2q}
\int_{t_0}^{T}\left\| \frac{\partial^{q+1} \nabla u_h(t)}{\partial t^{q+1}}\right\|_0^2 \ \di t.
\end{eqnarray}
To bound the first term on the right-hand-side, 
we apply \eqref{deri_dif_p} to $g=(I-P^r) u_{h}$, use Poincar\'e's inequality \eqref{poincare}, 
and the fact that for any function $v\in H_0^1(\Omega)$ it holds $\|\nabla(I-P^r)  v\|_0\le \|\nabla v\|_0$ to obtain
\begin{eqnarray*}
\lefteqn{
\sum_{n=0}^M (\Delta t)\|(I-P^r) u_{h,t}^n\|_0^2} \\
&\le &  C \sum_{n=1}^M (\Delta t)\|D((I-P^r)u_h^n)\|_0^2
 +C (\Delta t )^{2q} \int_{t_0}^T\left\|\frac{\partial^{q+1} (I-P^r)u_h(t)}{\partial t^{q+1}}\right\|_0^2 \ \di t\\
 &\le&  C C_p^2\sum_{n=1}^M (\Delta t)\|\nabla (I-P^r)Du_h^n\|_0^2
 +CC_p^2 (\Delta t )^{2q} \int_{t_0}^T\left\|\frac{\partial^{q+1} \nabla u_h(t)}{\partial t^{q+1}}\right\|_0^2 \ \di t.\nonumber
\end{eqnarray*}
Applying \eqref{cota_pod} to bound the first term on the right-hand side,
thereby taking the scaling of functions of $\mathbf U$ and $(M+1)/M\le 2$ into account,
yields
\begin{eqnarray}\label{mecu2}
 \sum_{n=0}^M (\Delta t)\|(I-P^r) u_{h,t}^n\|_0^2 \le C C_p^2\frac{T}{\tau^2} \sum_{k={r+1}}^{d_r} \lambda_k
 +CC_p^2 (\Delta t )^{2q} \int_{t_0}^T\left\|\frac{\partial^{q+1} \nabla u_h(t)}{\partial t^{q+1}}\right\|_0^2\ \di t.
\end{eqnarray}
From \eqref{mecu1} and \eqref{mecu2} we infer \eqref{eq:tra:pro_t_b}.
\end{proof}

\subsection{Time derivatives case}

The proof of the following  lemma is inspired by the arguments used for proving \cite[Lemma 3.1]{nos_der_pod}.

\begin{lema}\label{le:our_etal_tra}
There exist a positive constant~$C$ such that for~$X=L^2(\Omega)$ or~$X=H^1_0(\Omega)$, the following bound holds
for $z\in H^3(0,T;X)$ 
\begin{equation}\label{eq:zetast_tra}
\max_{0\le k\le {M} }\|z^k\|_X^2 \le  {C}\|\tilde z\|_X^2+ C\left(\frac{ T^2}{M}\sum_{n=0}^M \| z_t^n\|_X^2
+{T}(\Delta t)^4\int_{t_0}^{T}\left\|\frac{\partial z^3(t)}{\partial t}\right\|_X^2\ \di t\right),
\end{equation}
{where $\tilde z$ denotes either $\tilde z=z(t_0)$ or~$\tilde z=\overline z=\frac{1}{M+1}\sum_{j=0}^Mz^j$,  $z^n=z(t_n)$, $z_t^n=z_t(t_n)$}.
\end{lema}

\begin{proof} We first consider the case where $\tilde z=z^0$.
For each $k$ we have
$$
z^k=z^0+\int_{t_0}^{t_k}z_t(t) \ \di t.
$$
With the notation of Lemma~\ref{le:bosco}, adding and subtracting the Lagrange linear interpolant in time based on nodes
$t_{n-1}$ and $t_n$ we get
\begin{eqnarray}\label{first_tra}
z^k&=&z^0+\sum_{n=1}^k \int_{t_{n-1}}^{t_n} I_{2,t}z_t(t)\ \di t+\sum_{n=1}^k\int_{t_{n-1}}^{t_n} R_{2,t} z_t(t) \ \di t\nonumber\\
&=&z^0+\sum_{n=1}^k \left(\frac{\Delta t}{2}z_t(t_{n-1})+\frac{\Delta t}{2}z_t(t_n)\right)+\sum_{n=1}^k\int_{t_{n-1}}^{t_n} R_{2,t} z_t(t)\ \di t.
\end{eqnarray}
We now observe that, for the last term on the right-hand side of \eqref{first_tra}, applying the Cauchy--Schwarz 
inequality and the general bound \eqref{cota_inter} for the
Lagrange interpolant, in time instead of in space, 
for $x\in \Omega$ we obtain
\begin{eqnarray*}
\biggl|\int_{t_{n-1}}^{t_n} R_{2,t} z_t(t,x) \ \di t\biggr| &\le& \left(\int_{t_{n-1}}^{t_n} |R_{2,t} z_t(t,x)|^2 \ \di t\right)^{1/2}(\Delta t)^{1/2}\\
&\le& c_\text{\rm int} (\Delta t)^{5/2}\left(\int_{t_{n-1}}^{t_n} \left| \frac{\partial^3 z(t,x)}{\partial t}\right|^2 \ \di t\right)^{1/2},
\end{eqnarray*}
and consequently
\[
\left\|\int_{t_{n-1}}^{t_n} R_{2,t} z_t(t) \ \di t\right\|_X\le c_{\rm int}
(\Delta t)^{5/2}\left(\int_{t_{n-1}}^{t_n}\left\|\frac{\partial^3 z(t)}{\partial t}\right\|_X^2\ \di t
\right)^{1/2}.
\]
Taking norms in \eqref{first_tra} and using the above inequality we find
\begin{eqnarray*}
\|z^k\|_X&\le& \|z^0\|_X+\sum_{n=0}^k\Delta t \|z_t^n\|_X
+c_{\rm int}(\Delta t)^{5/2}\sum_{n=1}^k
\left(\int_{t_{n-1}}^{t_n}\left\|\frac{\partial^3 z(t)}{\partial t}\right\|_X^2\ \di t\right)^{1/2}
 \\
&\le&\|z^0\|_X+T^{1/2}(\Delta t)^{1/2}\left(\sum_{n=0}^M\|z_t^n\|_X^2\right)^{1/2}+c_{\rm int}{\sqrt{t_k}}(\Delta t)^2\left(\int_{t_0}^{T}\left\|\frac{\partial^3 z(t)}{\partial t}\right\|_X^2\ \di t
\right)^{1/2},
\end{eqnarray*}
from which we reach \eqref{eq:zetast_tra}.
The case where $\tilde z=\overline z$ follows easily by noticing that
\[
\overline z=\frac{1}{M+1}\sum_{j=0}^M\left(z^0+\int_{t_0}^{t_j}z_t(t)\ \di t\right) = z^0 + \frac{1}{M+1}\sum_{j=1}^M\int_{t_0}^{t_j}z_t(t)\ \di t,
\]
see \cite[Lemma~3.2]{nos_der_pod}.
\end{proof}

\begin{remark}
Arguing as in the proof of Lemma \ref{le:our_etal_tra} but changing $I_{2,t}$ by $I_{q,t}$ where $I_{q,t}$ is a Lagrange interpolant in time
based on nodes $t_{m-1}$, $t_m$,\ldots $t_{m+q-1}$ taking care that nodes include $t_n$ and are always part of the set $\left\{t_0,t_1,\ldots,t_N\right\}$  we obtain the following result for $q\ge 2$
\begin{equation}\label{eq:zetast_tra_p}
\max_{0\le k\le {M} }\|z^k\|_X^2 \le  {C}\|\tilde z\|_X^2+ C\left(\frac{ T^2}{M}\sum_{n=0}^M \| z_t^n\|_X^2
+{T}(\Delta t)^{2q}\int_{t_0}^{T}\left\|\frac{\partial^{q+1}z(t)}{\partial t^{q+1}}\right\|_X^2\ \di t\right),
\end{equation}
where the constant~$C$ only depends on~$q$ and~$c_{\rm int}$.
\end{remark}

\begin{lema}\label{le:maxPr_ut} There is a constant $C>0$ such that for each $q\ge 2$ the estimates
\begin{eqnarray}\label{max_dif_time_p}
\max_{0\le n\le M}\|u_h^n-P^ru_h^n\|_0^2&\le&  C\left(1+ \frac{T^2}{\tau^2}\right) C_p^2\sum_{k={r+1}}^{d_r}\lambda_k
+C T C_p^2(\Delta t)^{2q}\int_{t_0}^T\left\|\frac{\partial^{q+1}\nabla u_h(t)}{\partial t^{q+1}}\right\|_0^2 \ \di t,\nonumber\\
\max_{0\le n\le M}\|\nabla (u_h^n-P^ru_h^n)\|_0^2 &\le& C\left(1+ \frac{T^2}{\tau^2}\right) \sum_{k={r+1}}^{d_r}\lambda_k
+C T (\Delta t)^{2q}\int_{t_0}^T\left\|\frac{\partial^{q+1}\nabla u_h(t)}{\partial t^{q+1}}\right\|_0^2\  \di t.
\end{eqnarray}
hold. As a consequence, since $\Delta t =T/M$ and  $(M+1)/M\le 2$
\begin{equation}\label{max_dif_promedio_time_p}
{\Delta t }\sum_{n=0}^M \|u_h^n-P^ru_h^n\|_0^2\le  CC_p^2 \left(T \left(1+  \frac{T^2}{\tau^2}\right)\sum_{k={r+1}}^m\lambda_k +T^2(\Delta t)^{2q}\int_{t_0}^T\left\|\frac{\partial^{q+1}\nabla u_h(t)}{\partial t^{q+1}}\right\|_0^2 \ \di t\right).
\end{equation}
\end{lema}
\begin{proof}
The proof is achieved immediately by applying \eqref{eq:zetast_tra_p} to $z=(I-P^r)u_h$ with $X=L^2(\Omega)$ or $X=H_0^1(\Omega)$ together with Poincar\'e inequality \eqref{poincare} and \eqref{cota_pod_time_deri} and taking into account that
\begin{eqnarray*}
\int_{t_0}^{T}\left\|\frac{\partial^{q+1} (I-P^r) u_h(t)}{\partial t^{q+1}}\right\|_0^2 \ \di t&\le& C_p^2 \int_{t_0}^{T}\left\|\frac{\partial^{q+1} \nabla(I-P^r) u_h(t)}{\partial t^{q+1}}\right\|_0^2 \ \di t\\
&\le& C_p^2\int_{t_0}^{T}\left\|\frac{\partial^{q+1} \nabla u_h(t)}{\partial t^{q+1}}\right\|_0^2 \ \di t.
\end{eqnarray*}
\end{proof}

\begin{lema}\label{le:tra_cota_pro_time} There is a constant~$C>0$ such that 
for each $q\ge 2$ and for
 $1\le r\le d_r$ the following estimate holds
 \begin{eqnarray}\label{eq:tra:pro_time_p}
\int_{t_0}^{T} \|(I-P^r)u_h(t)\|_0^2\ \di t&\le&  CC_p^2\Biggl(T\left(1+\frac{T^2}{\tau^2}\right)\sum_{k={r+1}}^{d_r}\lambda_k
+{T^2}(\Delta t)^{2q}\int_{t_0}^T\left\|\frac{\partial^{q+1}\nabla u_h(t)}{\partial t^{q+1}}\right\|_0^2 \ \di t
\nonumber\\
&&+ (\Delta t)^{2q}
\int_{t_0}^{T}\left\| \frac{\partial^{q} \nabla u_h(t)}{\partial t^{q}}\right\|_0^2 \ \di t\Biggr).
\end{eqnarray}
\end{lema}
\begin{proof}
We apply \eqref{new_cota_int_orp} with $\varphi(s)=u_h(s)$ and then use \eqref{max_dif_promedio_time_p} for bounding $\sum_{n=0}^M (\Delta t)\|(I-P^r) u_h(t_n)\|_0$.
\end{proof}

\begin{lema}\label{le:tra_cota_pro_t_time}
For each $q\ge 2$, there is a constant~$C>0$ such that for
$1\le r\le d_r$ the following estimate holds

 \begin{equation}\label{eq:tra:pro_t_time}
\int_{t_0}^{T} \|(I-P^r)u_{h,t}(t)\|_0^2\ dt \le  CC_p^2\left(\frac{T}{\tau^2}\sum_{k={r+1}}^{d_r}\lambda_k
+(\Delta t)^{2q}\int_{t_0}^T\left\|\frac{\partial^{q+1}\nabla u_h(t)}{\partial t^{q+1}}\right\|_0^2 dt\right).
\end{equation}
\end{lema}
\begin{proof} 
The statement is obtained by applying \eqref{new_cota_int_orp} with $\varphi(s)=u_{h,t}(s)$ and then using  \eqref{cota_pod_time_deri} for bounding $\sum_{n=0}^M (\Delta t)\|(I-P^r) u_{h,t}(t_n)\|_0$.
\end{proof}

\section{Error analysis of the method}\label{section5}

We will consider the following semi-discrete POD-ROM approximation to approach \eqref{eq:model}: Find $u_r:(0,T]\rightarrow \bU^r$
such that
\begin{equation}\label{eq:pod}
(u_{r,t},v_r)+\nu(\nabla u_r,\nabla v_r)+(g(u_r),v_r)=(f,v_r),\quad \forall\ v_r\in \bU^r,
\end{equation}
with $u_r(0)=u_r^0\in \bU^r$ and $u_r^0\approx u^0$.

For the error analysis of this section, we will first assume that the nonlinear function
$g : \mathbb R \to \mathbb R$ is Lipschitz continuous (globally), with $g(u_r)$ being a short form  
notation of $g(u_r(t,x))$, $u_r(t,x)\in 
\mathbb R$, $t\in [0,T],\  x\in \Omega$. Then, we will extend the error analysis to a more general function.
Since we are dealing with a nonlinear equation, we need to prove some a priori bounds for the projection  of the Galerkin approximation over the reduced order space for performing the error analysis. 

\begin{lema} \label{lema12}
Assume $\frac{\partial^{2} u_h}{\partial t^2}\in H^1(0,T;H^1)$, $\Delta t \le C h^{1/3}$  and $h^{-1/2}\left(\sum_{k={r+1}}^{d_r}{\lambda_{k}}\right)^{1/2}\le C$. 
Then, the following bound holds
\begin{equation}\label{eq:cota_Pr}
\|P^r u_h(t)\|_\infty\le c_{\rm p,inf},\quad t\in [0,T].
\end{equation}
\end{lema}
\begin{proof}
We start proving bound \eqref{eq:cota_Pr} at the discrete times where 
we argue similarly as in \cite[(39)]{novo_rubino}.
For any $j=0,\ldots,{M}$, applying  \eqref{uhuht_inf}, \eqref{inv_b}, we obtain
\begin{equation}\label{como}
\|P^r u_h(t_j)\|_\infty\le \|u_h(t_j)\|_\infty+\|(I-P^r)u_h(t_j)\|_\infty\le c_{\inf}+C h^{-1/2}\|\nabla 
((I-P^r)u_h)(t_j)\|_0.
\end{equation}
To bound the last term on the right-hand side above, in the finite difference case, we apply \eqref{max_dif} and then the assumption on the tail of the eigenvalues gives
\begin{eqnarray*}
\|P^r u_h(t_j)\|_\infty \le c_{\inf}+C h^{-1/2}\left(\sum_{k={r+1}}^{d_r}\lambda_k\right)^{1/2}\le C.
\end{eqnarray*}
In the time derivative case, instead of \eqref{max_dif} we use \eqref{max_dif_time_p} with 
$q=2$, so that, using in addition the assumption on the length of the time step, we obtain
\begin{eqnarray*}
\|P^r u_h(t_j)\|_\infty &\le& c_{\inf}+Ch^{-1/2}\left(\sum_{k={r+1}}^{d_r}\lambda_k\right)^{1/2}\nonumber\\
&&+Ch^{-1/2}(\Delta t)^{{4}}\left(\int_0^T\left\|\frac{\partial^{3}\nabla u_h(t)}{\partial t^{3}}\right\|_0^2 \ \di t\right)^{1/2}\le C.
\end{eqnarray*}

Now, we want to extend the bound to any $t\in[0,T]$. To this end, 
we observe that for any $t\in (t_j, t_{j+1})$ we can decompose 
$$
  P^r u_h(t)=I_{2,t}(t)+R_{2,t}(t),
$$
where $I_{2,t}(t)$ is the linear  interpolant based on~$P_ru_h(t_{j})$ and~$P_ru_h(t_{j+1})$, and $R_{2,t}(t)=P_ru_h(t)-I_{2,t}(t)$.
Using the results from the first step of the proof gives 
\begin{eqnarray}
\label{12b1}
\|P^r u_h(t)\|_\infty &\le & \max\left\{ \|P^r u_h(t_j)\|_\infty,  \|P^r u_h(t_{j+1})\|_\infty\right\} + \left\|  R_{2,t}(t)\right\|_\infty
\nonumber \\
&\le& C+ \left\|  R_{2,t}(t)\right\|_\infty.
\end{eqnarray}
Now, taking into account that $R_{2,t}(t_j,\cdot)=R_{2,t}(t_{j+1},\cdot)=0$, applying \eqref{agmon}, and then the general bound \eqref{cota_inter} for the
Lagrange interpolant (in time instead of in space) yields
\begin{eqnarray*}
\left|R_{2,t}(t,\cdot) \right| &\le&  \left(\int_{t_j}^{t_{j+1}} \left|R_{2,t}(t,\cdot)\right|^2\ \di t\right)^{1/4}  \left(\int_{t_j}^{t_{j+1}} \left|\frac{\partial R_{2,t}(t,\cdot)}{\partial t}\right|^2\ \di t\right)^{1/4} 
\nonumber\\
& \le& 
 c_{\rm int} (\Delta t)^{3/2}\left(\int_{t_j}^{t_{j+1}} \left|\frac{\partial^2 P^r u_{h}(t,\cdot)}{\partial t^2}\right|^2\ \di t\right)^{1/2}.
\end{eqnarray*}
Applying the inverse inequality \eqref{inv_b} and the $H_0^1$ stability of the projection $P^r$ leads to 
\begin{eqnarray}
\label{12b2}
 \left\|  R_2(t,\cdot )\right\|_\infty &\le &  c_{\rm int} (\Delta t)^{3/2}\left(\int_{t_j}^{t_{j+1}}  \left\| 
 \frac{\partial^2 P^r u_{h}(t,\cdot)}{\partial t^2}\right\|_\infty^2\ \di t\right)^{1/2}
 \nonumber\\
&\le&  c_{\rm int} (\Delta t)^{3/2}Ch^{-1/2} 
\left( \int_{t_j}^{t_{j+1}}  \left\|  \frac{\partial^2 \nabla P^r u_{h}(t,\cdot)}{\partial t^2}\right\|_0^2\ \di t\right)^{1/2}
 \nonumber\\
&\le&  c_{\rm int} (\Delta t)^{3/2}Ch^{-1/2} 
 \left(\int_{t_j}^{t_{j+1}}  \left\| \frac{\partial^2 \nabla u_{h}(t)}{\partial t^2}\right\|_0^2\ \di t\right)^{1/2}.
 \end{eqnarray}
Combining \eqref{12b1} and~\eqref{12b2} with the assumptions of the lemma finishes the proof. 
\end{proof}

\begin{Theorem}\label{th1} Assume $g$ is Lipschitz continuous with Lipschitz constant~$L > 0$,
i.e., assume that $
|g(s) - g(t)| \le L |s-t|$. 
Let $u_r$ be the POD-ROM approximation solving \eqref{eq:pod} and let $P^r u_h$ be the $H_0^1$-orthogonal projection 
onto $\bU^r$ of the semi-discrete Galerkin approximation $u_h$ defined in \eqref{gal_semi}.
Then, for the constant $K$ given by~\eqref{constant_K} below, the following bound holds for all~$t\in [0,T]$
\begin{eqnarray}\label{er_pro_ur}
\lefteqn{\|u_r(t)-P^r u_h(t)\|_0^2+2\nu\int_0^t \|\nabla (u_r(s)-P^r u_h(s))\|_0^2 \ \di s}\nonumber\\
 &\le& e^{Kt}\|u_r(0)-P^r u_h(0)\|_0\\
 &&+ e^{KT}T\left(\int_0^t \|(I-P^r)u_{h,s}(s)\|_0^2\ \di s+L^2\int_0^t\|(I-P^r)u_h(s)\|_0^2 \ \di s\right).\nonumber
\end{eqnarray}
\end{Theorem}

\begin{proof}
A straightforward calculation, using the projection property, shows 
that the projection of $u_h$ satisfies for all $v_r\in \bU^r$
\begin{eqnarray*}
\lefteqn{(P^ru_{h,t},v_r)+\nu(\nabla P^ru_h,\nabla v_r)+(g(P^r u_h),v_r)}\\
&=&(f,v_r)+\left((P^r-I) u_{h,t},v_r\right) +\left(g(P^r(u_h))-g(u_h),v_r\right).
 \end{eqnarray*}
 Denoting the error by
 $
 e_r=u_r-P^r u_h \in \bU^r
 $
 and subtracting this identity from \eqref{eq:pod} gives for all $v_r\in \bU^r$
 \begin{equation}\label{eq:laer}
 (e_{r,t},v_r)+\nu(\nabla e_r,\nabla v_r) =((I-P^r)u_{h,t},v_r)+(g(P ^r u_h)-g(u_r),v_r)
 +(g(u_h)-g(P^r u_h),v_r).
 \end{equation}

 Taking $v_r=e_r$ in \eqref{eq:laer}, applying the Cauchy--Schwarz and Young inequalities together with
 the Lipschitz continuity of~$g$,
  we get
 \begin{eqnarray}\label{eq:th_er1}
\lefteqn{ \frac{1}{2}\frac{d}{dt}\|e_r\|_0^2+\nu\|\nabla e_r\|_0^2 }\nonumber \\
& \le & \frac{T}{2}\|(I-P^r)u_{h,t}\|_0^2
 +\frac{1}{2T}\|e_r\|_0^2+L\left\| e_r\right\|_0^2
 +L\left\|(I-P^r)u_h\right\|_0\left\|e_r\right\|_0\nonumber\\
&\le&\frac{T}{2}\|(I-P^r)u_{h,t}\|_0^2
 +\frac{1}{2T}\|e_r\|_0^2+L\|e_r\|_0^2 
 +\frac{T}{2}L^2\|(I-P^r)u_h\|_0^2+\frac{1}{2T}\|e_r\|_0^2.
  \end{eqnarray}
Multiplication by $2$ leads to 
\[
 \frac{d}{dt}\|e_r\|_0^2+2\nu\|\nabla e_r\|_0^2 \le K\|e_r\|_0^2
 +T\|(I-P^r)u_{h,t}\|_0^2+TL^2\|(I-P^r)u_h\|_0^2,
\]
where
  \begin{equation}
 \label{constant_K}
 K=\frac{2}{T} + 2L.
\end{equation}
Integrating in time and applying Gronwall's lemma we reach \eqref{er_pro_ur}.
 \end{proof}
 
 We now extend the error analysis to a general function that does not need to be globally Lipschitz continuous. Let $g$
 be the nonlinear term in \eqref{eq:model}. 
 We define the following quantities, with $c_{\rm p,inf}$ the constant in \eqref{eq:cota_Pr}:
 \begin{align}
 \label{L4}
 L_4&=\max\left\{ \frac{\left| g(s)-g(t)\right|}{\left|s-t\right|}\ \mid\  s\ne t,\quad \left|s\right|,\left| t\right|\le 4c_{\rm p,inf}\right\} ,\\
 \label{F4}
 F_4&=\max\left\{ \left| g(s)\right|\ \mid\  \left|s\right|\le 4c_{\rm p,inf}\right\}.
 \end{align}
 Let us consider  the following cut-off function
  \begin{eqnarray*}
 %\label{theta}
 \theta (s)=\left\{ \begin{array}{lcl} 1,& \quad& \left|s\right|\le 1,\\
\frac{1}{2}\left(1-\tanh\left(\frac{(s-3/2)}{(s-1)(2-s)}\right)\right),&&   1< \left|s\right|<  2,\\
 0,&& \left|s\right|\ge 2,
 \end{array}\right.
\end{eqnarray*}
which is a nonnegative $C^\infty$ function taking values $1$ for $\left|s\right|\le 1$ and~$0$ for $\left|s\right|\ge 2$

and define
\begin{equation}
\label{tilde_g}
\tilde g(s)= \theta\left(\frac{s}{2c_{\rm p,inf}}\right)g(s)
\end{equation}
which coincides with~$g$ for~$\left|s\right|\le 2c_{\rm p,inf}$ and vanishes for $\left|s\right|\ge 4c_{\rm p,inf}$. A simple calculation shows that $\left|\theta'(s)\right| \le 2$,
so that function~$\tilde g$ is Lipschitz-continuous in~${\mathbb R}$ and its Lipschitz constant is
\begin{equation}
\label{la_L}
L=L_4+F_4/c_{\rm p,inf},
\end{equation}
with $L_4$ and $F_4$ the constants in \eqref{L4}, \eqref{F4}.
\begin{Theorem}\label{th2} Let $P^r u_h$ be the $H_0^1$-orthogonal projection 
onto $\bU^r$ of the semi-discrete Galerkin approximation $u_h$ defined in \eqref{gal_semi}. Assume $ {u_h}\in H^{q+1}(0,T;H^1)$ for some $q\ge 2$. Let $C_q$ denote the maximum of the constants~$C$ in~\eqref{eq:tra:pro_b}, \eqref{eq:tra:pro_t_b}, \eqref{eq:tra:pro_time_p},  and \eqref{eq:tra:pro_t_time}.
For the constant~$L$ in~\eqref{la_L}, let  $K$ denote the value $K=2/T+2L$.  Assume $\Delta t$ is small enough so that
\begin{equation}
\label{condi2}
(\Delta t)^{2q}\left(L^2 \int_0^T\left\| \frac{\partial^q\nabla  u_h(t)}{\partial t^q}\right\|_0^2\ \di t  +
K_2 \int_0^T\left\| \frac{\partial^{q+1}\nabla  u_h(t)}{\partial t^{q+1}}\right\|_0^2\ \di t \right)\le e^{-KT} \frac{c_{\rm p,\inf}^2}{2TC_p^2 {c_{\rm inv}^2}}h^{d},
\end{equation}
where~$K_2=1$ in the finite difference case and $K_2=(1+L^2T^2)$ in the time derivative case.
 
Then, there exist $r_0\in [1, d_r]$  such that for $r>r_0$
 the POD approximation solving \eqref{eq:pod} with $u_r(0)=P^ru_h(0)$

satisfies the following bound for all $t\in[0,T]$
\begin{eqnarray}\label{eq:er_pro_ur_b}
\lefteqn{\|u_r(t)-P^r u_h(t)\|_0^2+2\nu\int_0^t \|\nabla u_r(s)-P^r u_h(s)\|_0^2 \ \di s\le 
  e^{Kt}TC_p^2\left(K_1\sum_{k={r+1}}^{d_r}\lambda_k\right)}\nonumber\\
 &&+
  e^{Kt}TC_p^2\left((\Delta t)^{2q} \left(L^2 \int_0^T\left\| \frac{\partial^q\nabla  u_h(t)}{\partial t^q}\right\|_0^2\ \di t  +
K_2 \int_0^T\left\| \frac{\partial^{q+1}\nabla  u_h(t)}{\partial t^{q+1}}\right\|_0^2\ \di t \right)\right), 
 \end{eqnarray}
 where
\begin{equation}
\label{K1}
 K_1=C_q\left(L^2T\left(1+\frac{T^2}{\tau^2}\right) + \frac{T}{\tau^2}\right).
\end{equation}
\end{Theorem}

\begin{proof}
Let~$\tilde u_r$ be the solution of the POD-ROM system~\eqref{eq:pod} with $\tilde g$ instead of~$g$ and initial condition $\tilde u_r(0)=P^r u_h(0)$. We first observe that the constant $c_{\rm inf}$ in \eqref{uhuht_inf} satisfies $c_{\rm inf}\le c_{\rm p,inf}$, compare \eqref{como}. Consequently, the solution  
$u_h$ of \eqref{gal_semi} solves also \eqref{gal_semi} with $\tilde g$ instead of
$g$ (since $\tilde g(u_h)=g(u_h))$. Applying Theorem~\ref{th1} with $\tilde g$ gives
\begin{eqnarray}\label{eq:er_pro_ur}
\lefteqn{\|\tilde u_r(t)-P^r u_h(t)\|_0^2+2\nu \int_0^t \left\|\nabla(\tilde u_r(s) -P^r u_h(s))\right\|_0^2\ \di s}\nonumber
\\
&\le &  e^{KT}T\left(\int_0^t \|(I-P^r)u_{h,s}(s)\|_0^2\ \di s+L^2\int_0^t\|(I-P^r)u_h(s)\|_0^2 \ \di s\right).
 \end{eqnarray}
 In the finite difference case we apply \eqref{eq:tra:pro_t_b}  to bound the first term on the
 right-hand side above and \eqref{eq:tra:pro_b} to bound the second one.
 In the time derivative case we use  \eqref{eq:tra:pro_t_time} and  \eqref{eq:tra:pro_time_p} to bound the first and
 second term, respectively.  Then, we obtain an estimate of form \eqref{eq:er_pro_ur_b}, but with $\tilde u_r$ instead of $u_r$:
 \begin{eqnarray}\label{bb_ii}
\lefteqn{\|\tilde u_r(t)-P^r u_h(t)\|_0^2+2\nu\int_0^t \|\nabla \tilde u_r(s)-P^r u_h(s)\|_0^2 \ \di s\le 
  e^{Kt}TC_p^2\left(K_1\sum_{k={r+1}}^{d_r}\lambda_k\right)}\nonumber\\
 &&+
  e^{Kt}TC_p^2\left((\Delta t)^{2q} \left(L^2 \int_0^T\left\| \frac{\partial^q\nabla  u_h(t)}{\partial t^q}\right\|_0^2\ \di t  +
K_2 \int_0^T\left\| \frac{\partial^{q+1}\nabla  u_h(t)}{\partial t^{q+1}}\right\|_0^2\ \di t \right)\right).\quad
 \end{eqnarray}
 
For obtaining \eqref{eq:er_pro_ur_b} it is sufficient to find 
a condition such that $\left\|\tilde u_r(t)\right\|_\infty \le 2c_{\rm p,inf}$ for all $t\in [0,T]$,
since in that case $\tilde g(\tilde u_r)=g(\tilde u_r)$ and then
$\tilde u_r$ is the solution of \eqref{eq:pod} with initial condition $P^r u_h(0)$.
Using the triangle inequality
$$
\left\| \tilde u_r(t)\right\|_\infty^2\le 2\left\| P^r u_h(t)\right\|^2_\infty+2\left\| \tilde u_r(t)-P^r u_h(t)\right\|_\infty^2
$$
and applying \eqref{eq:cota_Pr}, the inverse inequality \eqref{inv}, \eqref{bb_ii} and \eqref{condi2}  we deduce that
\begin{equation}
\label{almost}
\left\| \tilde u_r(t)\right\|_\infty^2 \le 2 c_{\rm p, inf}^2 + {c_{\rm p, inf}^2}+2c_{\rm inv}h^{-d}e^{KT}T C_p^2\left(K_1\sum_{k={r+1}}^{d_r}\lambda_k\right).
\end{equation}
Taking $r_0$ such that
$$
2c_{\rm inv}h^{-d}e^{KT}T C_p^2\left(K_1\sum_{k={r_0+1}}^{d_r}\lambda_k\right)\le c_{\rm p,inf}^2
$$
finishes the proof.
\end{proof}
\begin{remark}
  Let us observe that the previous theorem requires $h^{-d/2}\left(\sum_{k={r+1}}^{d_r}{\lambda_{k}}\right)^{1/2}\le C$
  which is a stronger condition than the one assumed in Lemma \ref{lema12} for $h\le 1$. Let us also observe that condition \eqref{condi2} holds for $q= 2$ whenever $\Delta t \le C h^{d/4}$ which is also a stronger time step restriction, for $h < 1$, than that assumed in Lemma~\ref{lema12}. This means that the
  assumptions needed to prove the a priori bounds for the projection $P^r u_h$ in the $L^\infty$ norm are already included into
  the assumptions of Theorem \ref{th2}. 
  \end{remark}
\begin{Theorem}\label{th3}
Under assumptions of Theorem~\ref{th2} there exists a constant $C>0$, depending on $T$, $C_q$ and the constant~$L$ in~\eqref{la_L},
such that the  following bound is valid
\begin{eqnarray}\label{max_puntual}
\max_{0\le n\le M} \|u_r(t_n)-u_h(t_n)\|_0^2&\le& 
   C \sum_{k={r+1}}^m\lambda_k\\
&&+C(\Delta t)^{2q}\int_{0}^{T}\left( \left\|\frac{\partial^{q} \nabla u_h(t)}{\partial t^{q}}\right\|_0^2+\left\| \frac{\partial^{q+1} \nabla u_h(t)}{\partial t^{q+1}}\right\|_0^2 \right)\ \di t.\nonumber
\end{eqnarray}
Moreover, for all $t\in[0,T]$ assuming also $\frac{\partial^{q} u_h}{\partial t^{q}}\in L^\infty([0,T];H^1)$, $q\ge 2$, it holds that 
\begin{eqnarray}\label{max_puntual_todo}
\|u_r(t)-u_h(t)\|_0^2&\le& 
   C \left(\sum_{k={r+1}}^m\lambda_k+(\Delta t)^{2q}
\int_{0}^{T}\left\| \frac{\partial^{q+1} \nabla u_h(t)}{\partial t^{q+1}}\right\|_0^2 \ \di t\ \right)
\\\nonumber&&+C (\Delta t)^{2q} \max_{0\le t\le T}\left\| \frac{\partial^{q} \nabla u_h(t)}{\partial t^{q}}\right\|_0^2.
\end{eqnarray}
\end{Theorem}
\begin{proof}
Since 
\begin{equation}\label{decom_ur_uh}
u_r(t)-u_h(t)\le \left(u_r(t)-P^r u_h(t)\right)+\left(P^r u_h(t)-u_h(t)\right)=e_r(t)+\left(P^r u_h(t)-u_h(t)\right)
\end{equation}
we can apply Theorem \ref{th2} with $t=t_n$ to bound the first term on the right-hand side. 
The second term is bounded by \eqref{max_dif} or \eqref{max_dif_time_p} in the
finite difference or time derivative case, respectively. Combining all bounds gives 
\eqref{max_puntual}.

To prove \eqref{max_puntual_todo} we also apply Theorem~\ref{th2} to bound the first term. For the second term 
we argue as in Lemma~\ref{le:bosco}, compare also Remark~\ref{rem:rem1}, for estimating $\|(I-P^r)\varphi(t)\|_0$, $t\in[0,T]$, in terms of the values at the discrete set of points. 

Assume $t\in[t_n,t_{n+1}]$.
We take the Lagrange interpolant in time based on $q$ nodes, $t_m,t_{m+1},\ldots,t_{m+q-1}$ (with $t_n$, $t_{n+1}$ being
part of the nodes and 
taking care that
the nodes are always a subset of $\left\{t_0,\ldots,t_M\right\}$) and 
argue as in \eqref{eq:bosco2} and \eqref{eq:bosco3}, 
but bounding the $L^\infty(t_n,t_{n+1})$ norm instead of the $L^2(t_n,t_{n+1})$ norm, 
to get 
\begin{eqnarray}\label{new_cota_int_p}
\|(I-P^r)\varphi(t)\|_0^2 &\le& C\sum_{l=1}^q\|(I-P^r) \varphi(t_{m+1-l})\|_0^2+C (\Delta t)^{2q}
\max_{0\le t\le T}\left\| \frac{\partial^q \nabla \varphi(t)}{\partial t^{q}}\right\|_0^2\nonumber\\
&\le& C q \max_{0\le n\le M}\|(I-P^r) \varphi(t_n)\|_0^2+C (\Delta t)^{2q}
\max_{0\le t\le T}\left\| \frac{\partial^q \nabla \varphi(t)}{\partial t^{q}}\right\|_0^2.
\end{eqnarray}
Applying \eqref{new_cota_int_p} to $\varphi=u_h$ and utilizing \eqref{max_dif} in the
finite difference case yields
\begin{equation}\label{eq:I_Pr_u_h_dif}
\|(I-P^r)u_h(t)\|_0^2 \le C \sum_{k={r+1}}^m\lambda_k+C (\Delta t)^{2q} \max_{0\le t\le T}\left\| \frac{\partial^{q} \nabla u_h(t)}{\partial t^{q}}\right\|_0^2,\quad t\in[0,T].
\end{equation}
In the time derivative case, we obtain with \eqref{max_dif_time_p}
\begin{eqnarray}\label{eq:I_Pr_u_h_time}
\|(I-P^r)u_h(t)\|_0^2 &\le& C \sum_{k={r+1}}^m\lambda_k+C(\Delta t)^{2q}\int_{0}^{T}\left\| \frac{\partial^{q+1} \nabla u_h(t)}{\partial t^{q+1}}\right\|_0^2 \ \di t\nonumber\\
&&+C (\Delta t)^{2q} \max_{0\le t\le T}\left\| \frac{\partial^{q} \nabla u_h(t)}{\partial t^{q}}\right\|_0^2,\quad t\in[0,T].
\end{eqnarray}
Inserting \eqref{eq:I_Pr_u_h_dif} and \eqref{eq:I_Pr_u_h_time} into \eqref{decom_ur_uh} we conclude \eqref{max_puntual_todo}.
\end{proof}
 
Since time integrators approximate the solution of the corresponding continuous-in-time problem, with a temporal error being proportional to powers of the time step used,
the essential consequence of the presented analysis is that any time integrator in the FOM method and any time integrator in the POD-ROM method can be used,
their errors being added to those studied above, which depend on the tail of the eigenvalues and
on $\Delta t$, i.e., the distribution of the snapshots in time. 
The influence of $\Delta t$ is determined by  how many bounded time derivatives $u_h$ has and by the
size of the time derivatives, since we can apply \eqref{max_puntual_todo} for different values of $q$. 

The error estimate \eqref{max_puntual_todo} implies that we can approximate the continuous equation at any time with the reduced order model although we have computed the snapshots only in a discrete set of times.
It will be shown in the next section, with some numerical examples, that a small number of snapshots, corresponding to a small number of discrete times, is enough to accurately reproduce the errors in the full interval.
 
\section{Numerical studies}\label{sec:numres}

 We consider the system known as the Brusselator with diffusion
\begin{equation}
\label{bruss}
\begin{array}{rclcl}
u_t&=&\nu\Delta u +1 +u^2v  - 4 u,&\qquad& (t,x)\in \Omega \times (0,T],\\
v_t &=&\nu \Delta v+3 u -u^2v,& \qquad
& (t,x)\in \Omega \times (0,T],\\
%&& u(x,0)= u_0(x), \quad v(x,0)=v_0(x),&& x\in \Omega,\\
&&u(x,t)=1, \quad v(x,t) =3,&& (t,x)\in \Gamma_1\times(0,T],\\
&&\partial_n u(x,t)=\partial_n v(t,x) =0,&&  (x,t)\in \Gamma_2\times(0,T],\\
\end{array}
\end{equation}
where $\nu$ is a positive parameter, $\Omega=[0,1]\times[0,1]$, $\Gamma_1\subset\partial\Omega$ is the union of of sides
$\{x=1\}\bigcup \{ y=1\}$, and~$\Gamma_2$ is the rest of~$\partial\Omega$. This system has an unstable equilibrium $u=1$, $v=3$, and, for $\nu$, sufficiently small, a stable limit cycle (see e.g., \cite[\S~I.16,\S~2.10]{Hairer}). If the boundary conditions are $\partial_n u(x,t)=\partial_n v(x,t) =0$ for $(t,x)\in(0,T]\times\partial\Omega$, then, the periodic orbit is flat in space, $\nabla u=\nabla v=0$, but with those in~\eqref{bruss} it has more spatial complexity. This can be seen in Fig.~\ref{solus}, where we show the solution at the time where the maximum of~$\left\|\nabla u\right\|_0^2+\left\| \nabla v\right\|_0^2$ is achieved.
\begin{figure}[h]
\begin{center}
\includegraphics[height=4truecm]{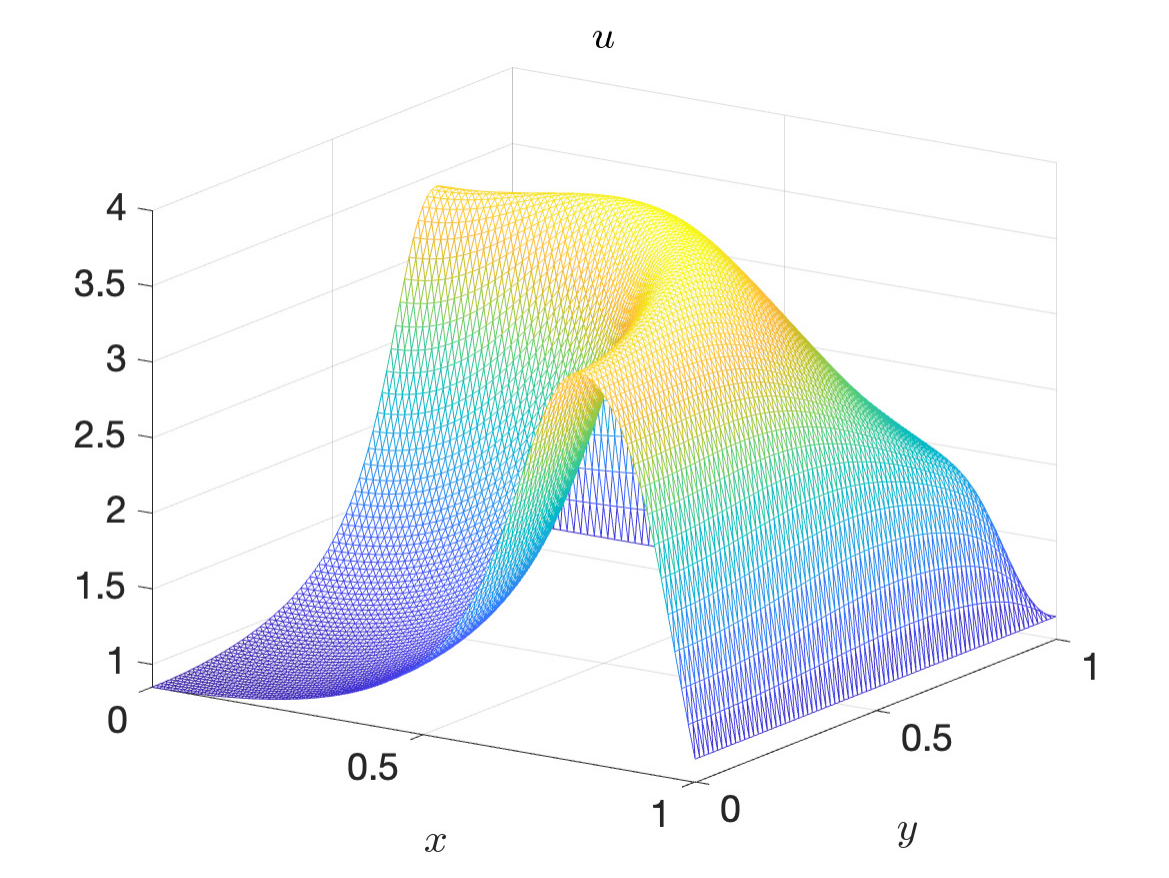}
\quad
\includegraphics[height=4truecm]{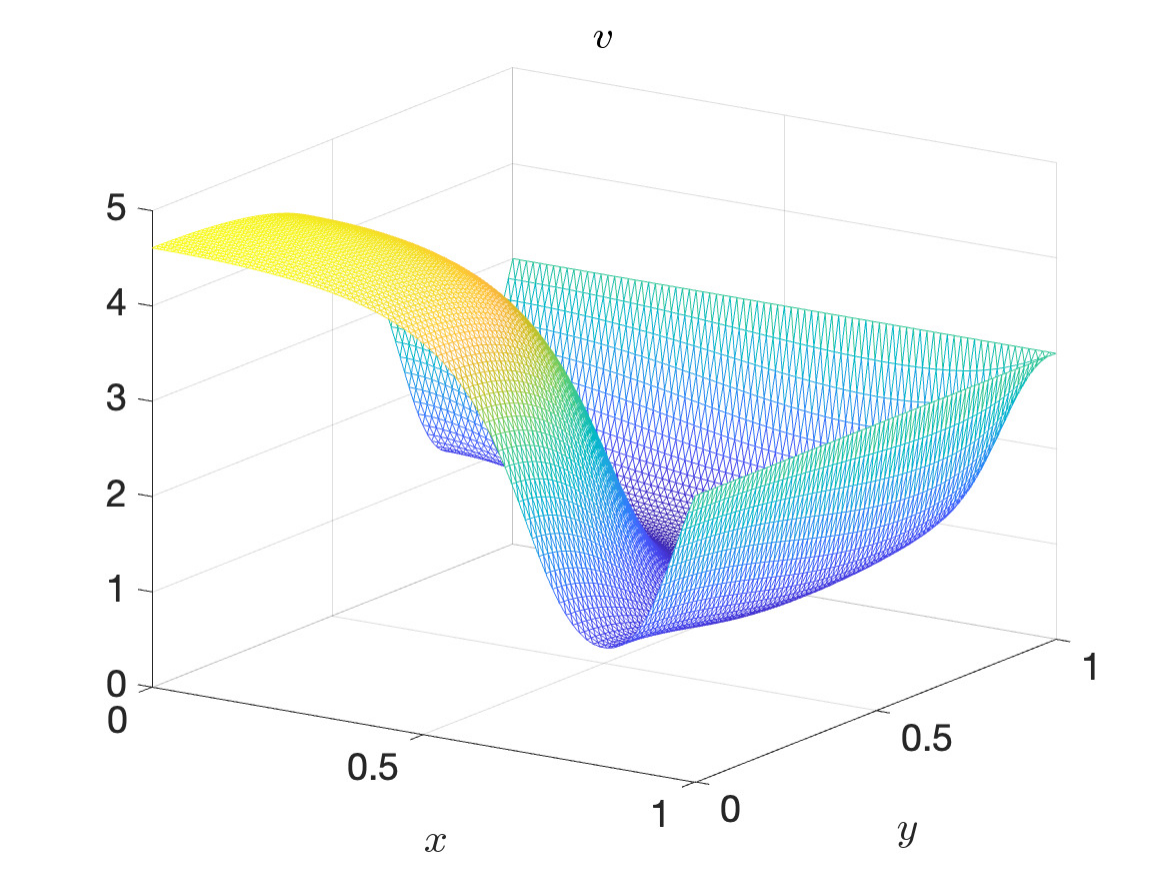}
\begin{caption}{\label{solus} Components $u$ (left) and~$v$ (right) of the periodic solution at time $t=6.5159$ where the maximum of~$\left\|\nabla u\right\|_0^2+\left\| \nabla v\right\|_0^2$ is achieved.}
\end{caption}
\end{center}
\end{figure}

For several values of~$\nu$ we computed finite element approximations to this limit cycle, and took different sets of snapshots at equally spaced times from it. We summarize here the results for $\nu=0.002$, which offered an adequate balance between spatial complexity and computability. In this case we consider a triangulation of~$\Omega$ based on a uniform mesh with 80  subdivision per side and with diagonals running southwest-northeast. We used a spatial discretization with quadratic finite elements, which resulted in a system with 50558 degrees of freedom. For this discretization, we computed the corresponding limit cycle as the zero of a return map to a Poincar\'e section, the corresponding set of equations being solved by a Newton--Krylov method~(see \cite{Joan-periodic} for details), the period being~$T=7.090636$ (up to 7 significant digits).
For the time integration of the corresponding system of ordinary differential equations (ODEs) we used the numerical differentiation formulae (NDF) \cite{NDF1}, 
in the variable-step, variable-order implementation of {\sc Matlab}'s command {\tt ode15s} \cite{odesuite}, with sufficiently small tolerances for the local errors (below $10^{-10}$).

To have an idea of the accuracy of the (FEM approximation to the) computed orbit $\bu_h=(u_h,v_h)$, we also computed an approximation~$\bu_{h/2}=(u_{h/2},v_{h/2})$ to the periodic orbit on a mesh with 160 subdivisions on each side of $\Omega$,
to be considered as reference solution, and measured the difference between them over a period. The corresponding values are shown in Fig.~\ref{errsh}, where we see that they vary by an order of magnitude both in the~$L^2$ and~$H^1$ norms. The maximum of these errors is $5.39\times 10^{-5}$ in the $L^2$ norm and~$0.0273$ in the $H^1$ norm (both achieved at time~$t=6.9591$, a little after the maximum gradients are achieved). Because of the maximum (estimation of the) error achieved in the~$H^1$ norm, in the experiments below and unless stated otherwise, the dimension $r$ of the POD basis was the minimum for which,
\begin{equation}
\label{r_select}
\left(\sum_{k=r+1}^{d_r} \lambda_k\right) \le 0.0273,
\end{equation}
since larger values of~$r$ may produce smaller errors $\left\| \nabla(\bu_r - P^r\bu_h)\right\|_0$ while not improving the error~$\left\| \nabla(\bu_r - \bu)\right\|_0$.
\begin{figure}[t]
\begin{center}
\includegraphics[height=2.3truecm]{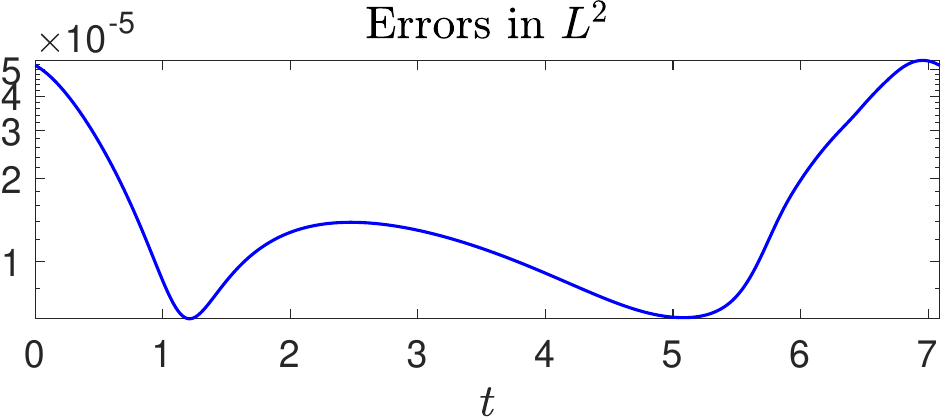}
\quad
\includegraphics[height=2.3truecm]{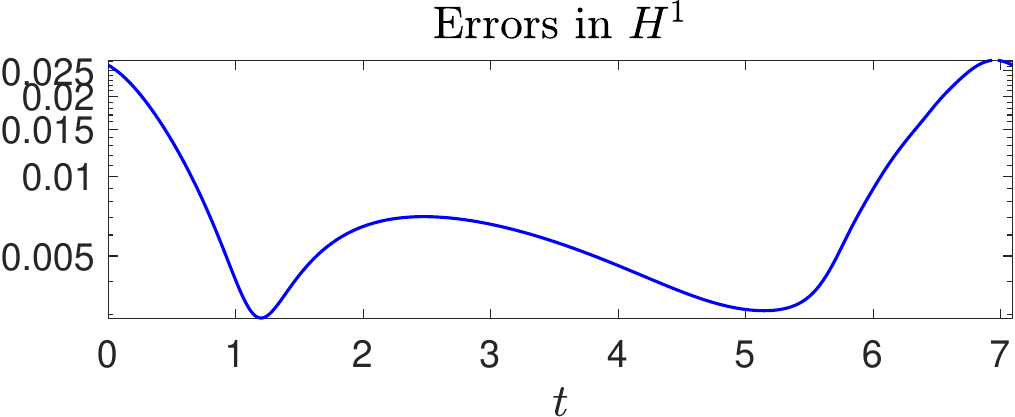}
\begin{caption}{\label{errsh} Estimation of the errors in~$L^2$ (left) and~$H^1$ (right) of the FEM approximation~$(u_h,v_h)$ to the periodic orbit solution of~\eqref{bruss}.}
\end{caption}
\end{center}
\end{figure}

{
It is interesting to check the size of different time derivatives, since these have a direct effect on the size of the errors of different approximations, see for example \eqref{max_puntual_todo}. In Fig.~\ref{derivatives} we show the $L^2$ norms of the time derivatives from the second to the fifth order. It can be seen that they attain their maximum value at the end of the period. In the figures below depicting errors of the POD method, we will see that peak errors are found also near the end of the period, where time derivatives of the solution are larger. We also see that the higher the order of the derivative, the larger the variation in size, the fifth derivative varying by five orders of magnitude between its smallest and largest values. 
\begin{figure}[h]
\begin{center}
\includegraphics[height=2.3truecm]{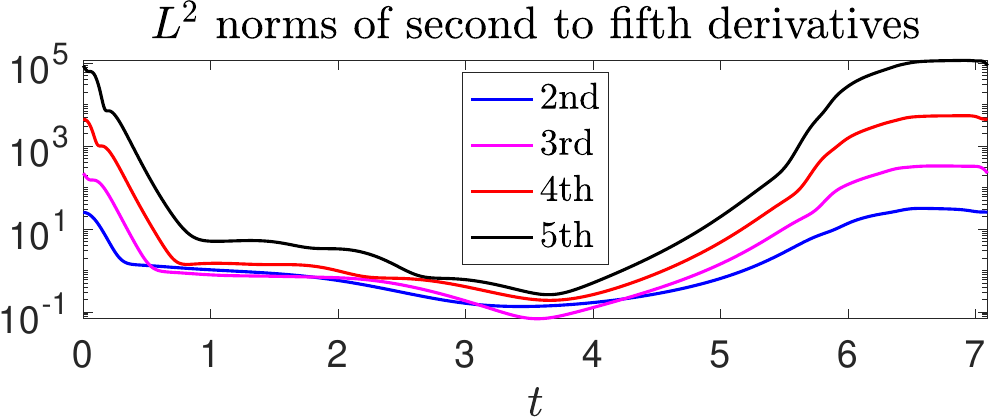}
\begin{caption}{\label{derivatives}  $L^2$ norms of the time derivatives of the finite element approximation $\bu_h$ to the periodic orbit solution of~\eqref{bruss}.}
\end{caption}
\end{center}
\end{figure}
}

We first check that, as stated by  Theorem~\ref{th3}, the POD method produces good continuous-in-time approximations. For this purpose, we computed the POD in finite differences case with $\Delta t=T/128$, where $T$ is the period (by trying with different values of $\Delta t$ we checked that this was enough for this purpose) and consider the POD method with $r=33$, selected as indicated above. Then, to simulate continuity in time we measured errors every T/2048 units of time (16 times denser than that of the snapshots). In Fig.~\ref{fig3}
 \begin{figure}[h]
 \begin{center}
\includegraphics[height=2.7truecm]{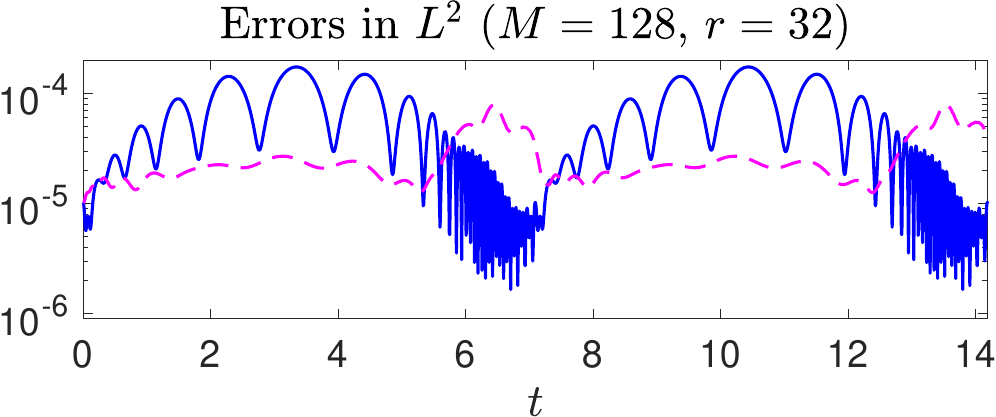}
%\quad
\includegraphics[height=2.7truecm]{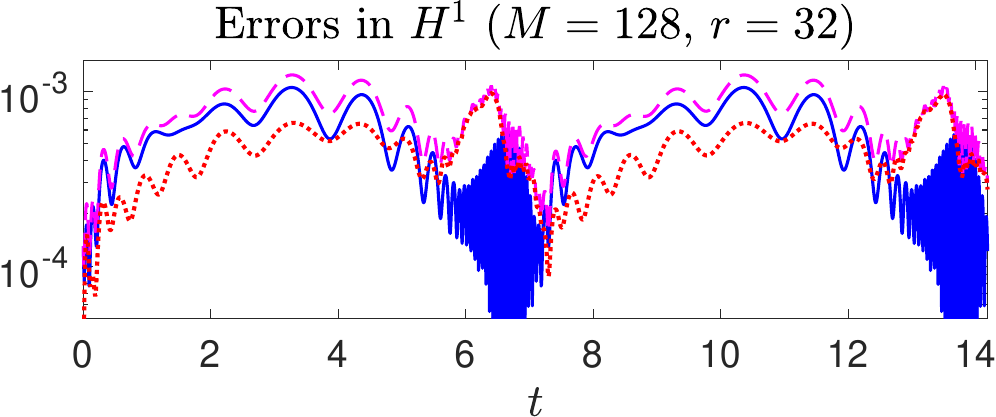}
\begin{caption}{\label{fig3} $L^2$ (left) and $H^1$ (right) errors  $\bu_h-P_r\bu_h$ (continuous blue line), $\bu_r-\bu_h$ (discontinuous line) and $\bu_r-P_h\bu_h$ (right plot only, red dotted line), for~ $M=128$, {$r=32$}. Errors measured every $T/2048$ units of time over two periods ($T$ being the period).}
\end{caption}
\end{center}
\end{figure}
we show the errors $\bu_h-P_r\bu_h$ (continuous line) and~$\bu_h-\bu_r$ (discontinuous line) in the $L^2$ norm (left plot) and in the $H^1$ norm (right plot) measured over two periods at every $T/2048$ units of time. We can see that the errors are small not only at the 128 times coinciding with those of the snapshots, but also on a grid 16 times denser over the two periods, confirming what Theorem~\ref{th3} and \eqref{eq:I_Pr_u_h_dif}, \eqref{eq:I_Pr_u_h_time} state.
We can also see that the errors $\bu_h-\bu_r$ become periodic after one period of~$\bu_h$, in both the $L^2$ and $H^1$ norms. We also notice that the errors $\left\| \bu_h-\bu_r\right\|_0$ (discontinuous line on the left plot) are much smaller than $\left\| \bu_h-P^r\bu_h\right\|_0$ (continuous line) over most of the period, %a phenomenon for which, at present, we do not have an explanation. 
{which is, in our opinion, it could be caused by the fact that the norm in which the error is measured 
is based on a different inner product than the projection that is used in the POD.}
On the right plot, however, $\left\|\nabla( \bu_h-\bu_r)\right\|_0$ is always above~$\left\| \nabla(I-P^r)\bu_h\right\|_0$, as it should be, since the latter is the best approximation error.
 On the right-plot, we also show the error
$\left\|\nabla(\bu_r-P^r\bu_r)\right\|_0$ (red dotted-line), which was estimated in~Theorem~\ref{th2}.

We now check the influence of a larger spacing between snapshots. In Fig.~\ref{fig4} we show
 \begin{figure}[h]
 \begin{center}
\includegraphics[height=2.7truecm]{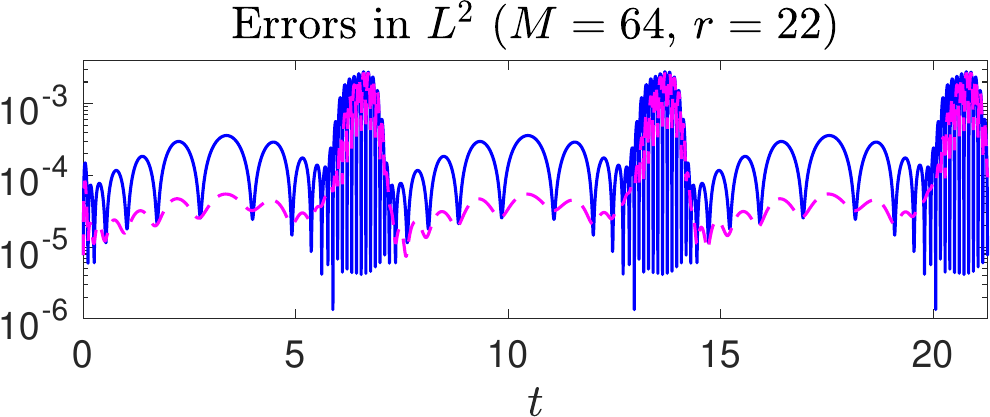}
%\quad
\includegraphics[height=2.7truecm]{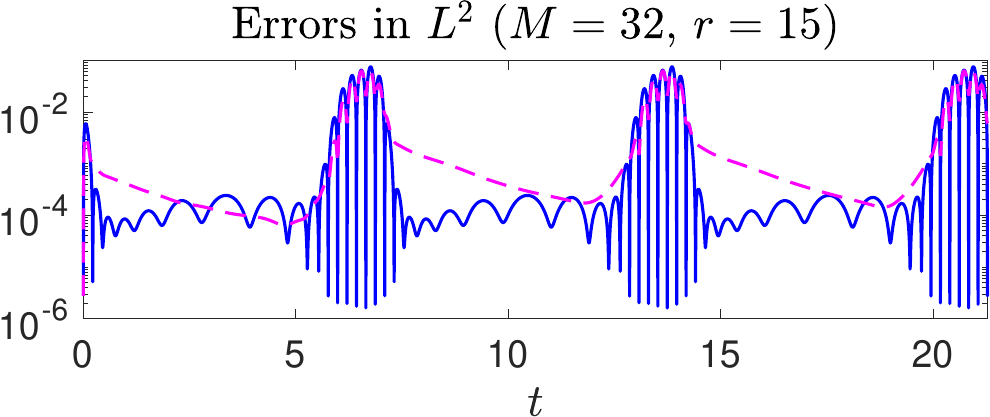}
\begin{caption}{\label{fig4} Errors $\left\|\bu_h-P_r\bu_h\right\|_0$ (continuous blue line), $\left\|\bu_r-\bu_h\right\|_0$ (discontinuous line) for~$M=64$, $r=22$ on the left and $M=32$, $r=15$ on the right.  Errors measured every $T/2048$ time units over three periods ($T$ being the period).}
\end{caption}
\end{center}
\end{figure}
$\left\|\bu_h-P_r\bu_h\right\|_0$ (continuous blue line), $\left\|\bu_r-\bu_h\right\|_0$ (discontinuous line) for~$M=64$ (left plot) and~$M=32$ (right plot), with errors measured every~$T/2048$ units of time, as before but on three periods. Comparing with the left plot in Fig.~\ref{fig3} we notice that in~Fig.~\ref{fig4}
the largest errors are considerably larger in Fig.~\ref{fig4}, where both errors reach the values $0.0028$ (left plot) and $0.0751$ (right-plot), than in~Fig.~\ref{fig3}, where they stay below $1.75\times 10^{-4}$. If, with $M=64$ and~$M=32$ we use the same value of~$r$ as with $M=128$, that is, $r=32$ the errors hardly decrease to~$0.0025$ and $0.0242$ for $M=64$ and~$M=32$, respectively, suggesting that the larger values of~$\Delta t$ is the factor limiting accuracy.
Also, comparing with Fig.~\ref{fig3},  we notice that in Fig.~\ref{fig4} it takes longer for the periodic behavior to be seen, another indication that POD approximation worsen if $\Delta t$ becomes too large.

To check further values of spacing between snapshots, we measured the errors every $T/2048$ units of time for different values of $M$. The results are shown in Table~\ref{table1}, where we show the maximum errors over four periods.
\begin{table}[h]
$$
\begin{array}{|r|c|c|c|c|c|}
\hline
\vphantom{\Big|}
M\  & r& \left\| \bu_h-\bu_r\right\|_0 & \left\| \bu_h-P^r\bu_h\right\|_0 & \left\|\nabla (\bu_h-\bu_r)\right\|_0 &  \left\|\nabla (I-P^r)\bu_h \right\|_0 \\ \hline
\vphantom{\big|} 512 &  35&   7.805\times 10^{-5} &    9.907\times 10^{-5}  &   8.576\times 10^{-4}  &   7.201\times 10^{-4} \\ \hline
\vphantom{\big|} 256  &  34&     7.446\times 10^{-5} &    9.456\times 10^{-5}  &   8.338\times 10^{-4}  &   7.026\times 10^{-4} \\ \hline
\vphantom{\big|} 128  & 32&     8.490 \times 10^{-5} &    1.745\times 10^{-4}  &   1.239\times 10^{-3}  &   1.050\times 10^{-3} \\ \hline
\vphantom{\big|} 64  &  22&      2.677\times 10^{-3} &     2.774\times 10^{-3}  &   1.210\times 10^{-1}  &   1.195\times 10^{-1} \\ \hline
\vphantom{\big|} 32  &  15&     6.268\times 10^{-2} &     7.508\times 10^{-2}  &   1.650\ \   &   1.593 \ \  \\ \hline
\end{array}
$$
\caption{\label{table1} 
Maximum errors for different values of $M$ measured every $T/2048$ units of time for four periods. Finite difference case.
}
\end{table}
We see that whereas for $M\ge 128$ the errors are almost independent of $\Delta t$ (to reduce them further the value of~$r$ should be increased) this is not the case for $M<128$, where every time the space between snapshots is doubled, errors multiply approximately by a factor of~$10$. Thus, of the different terms on the right-hand side of \eqref{max_puntual_todo}
and \eqref{eq:I_Pr_u_h_dif} those with powers of~$\Delta t$ become dominant for~$M< 128$.

One may wonder why the minimum value of~$r$ satisfying~\eqref{r_select} decreases when $\Delta t$ increases. The reason is that the larger the number of snapshots (the smaller the value of~$\Delta t$) the slower the decay of the eigenvalues~$\lambda$, as it can be checked on the left plot in Fig.~\ref{figsigmas},
\begin{figure}[h]
\begin{center}
\includegraphics[height=4truecm]{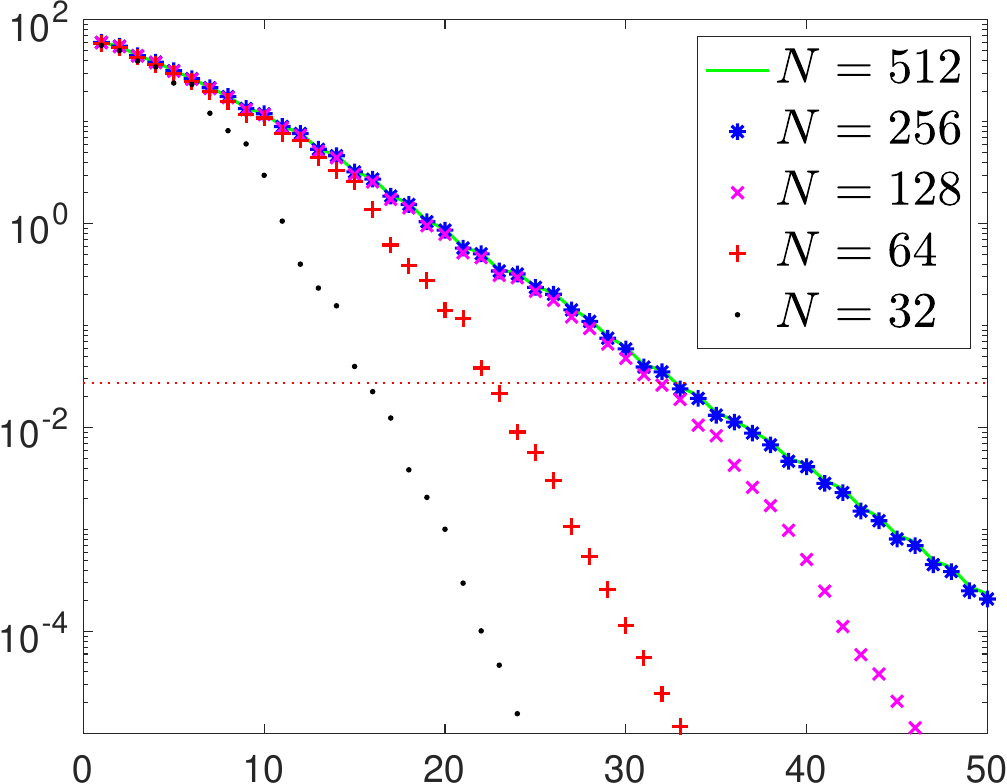}\quad \includegraphics[height=4truecm]{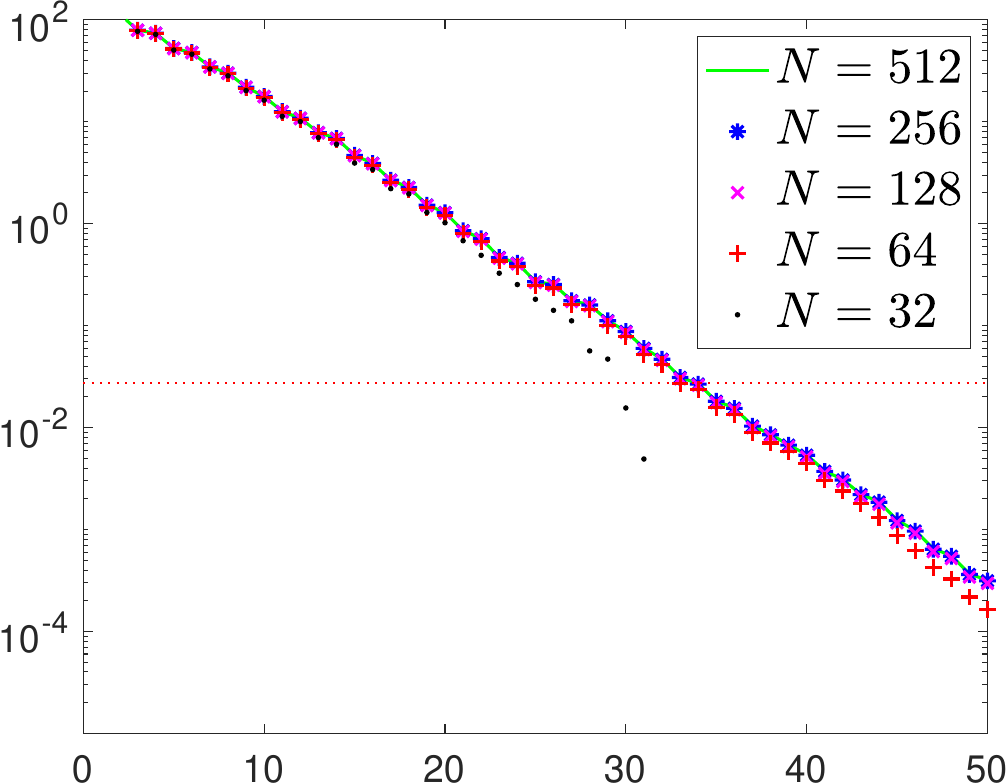}
\caption{\label{figsigmas}Singular values $\sigma_k=\sqrt{\lambda_k}$ for different values of~$M$. Left: equally-spaced times. Right, times distributed according to~\eqref{equi_formula}.}
\end{center}
\end{figure}
where we show singular values~$\sigma_k=\sqrt{\lambda_k}$ for different values of~$M$. The horizontal red dotted line marks $0.0273$. The contrast between Figs.~\ref{fig3} and~\ref{fig4} suggested us the idea that perhaps with nonuniform $\Delta t$ concentrating snapshots where larger errors are,  it might be possible to get the same rate of decay for different values of~$M$. This suggestion proved true as it can be seen in the right-plot in Fig.~\ref{figsigmas}, where  we show the singular values when the times $t_j$ are chosen so that
\begin{equation}
\label{equi_formula}
\int_{t_{j-1}}^{t_j} \left\| \nabla \bu_{h,t}(t,\cdot)\right\|_0^2 \,dt = \hbox{\rm constant},\qquad j=1,\ldots,M.
\end{equation}
It can be seen that, except for~$M=32$, they decay with very similar rates. 
We may also comment that
with this distribution of times for the snapshots, the maximum of~$\| \bu_h-\bu_r\|_0$   is  $4.1\times 10^{-4}$ for $M= 128$, $6.8\times 10^{-4}$ for $M=64$ and~$2.4\times 10^{-3}$ for $M=32$, these last two values being much smaller than those in~Fig.~\ref{fig4} (compare also with the same values of~$M$ in~Table~\ref{table1}).  However, the case of unevenly distributed times to place the snapshots will be studied elsewhere.

The fact that in Fig.~\ref{fig4} larger errors concentrate only in one part of the period while remaining much smaller the rest of the time suggests that another way of keeping errors small while using fewer snapshots may be to use patches of  uniform grids with smaller values of~$\Delta t$ only when errors are large, or, in practice, when time derivatives of the FEM approximation are large (recall Fig.~\ref{derivatives}). To test this idea (for which the analysis in the present paper is easy to adapt) we consider the case where we use $\Delta t=T/128$ in~$[0,T/32]\cup [(23/32)T,T]$ and larger $\Delta t$ in~$[T/32,(23/32)T]$. In Fig.~\ref{fig_patches} we show the results when,  in~$[T/32,(23/32)T]$,  we take $\Delta t=T/64$ (left), which gives $M+1$ snapshots with $M=84$ and~$\Delta t=T/32$ (right), which gives $M+1$ snapshots with $M=62$. The scale of the axes is that of the left plot in~Fig.~\ref{fig4} for better comparison.
 \begin{figure}[h]
 \begin{center}
\includegraphics[height=2.7truecm]{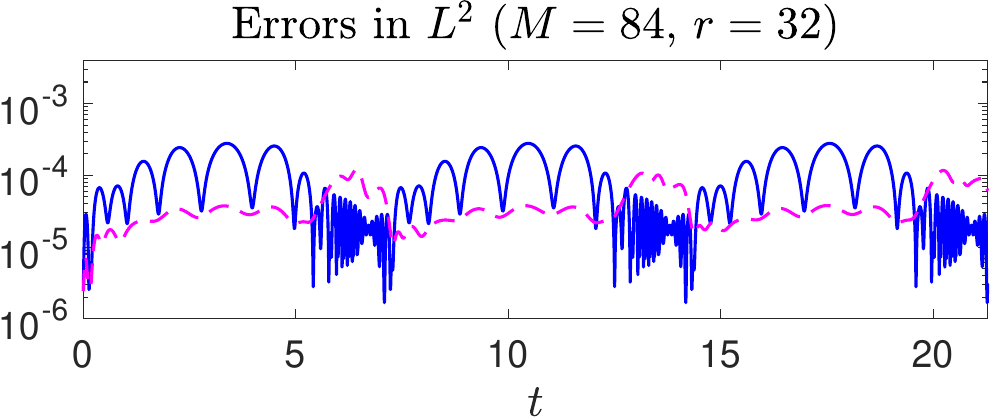}
%\quad
\includegraphics[height=2.7truecm]{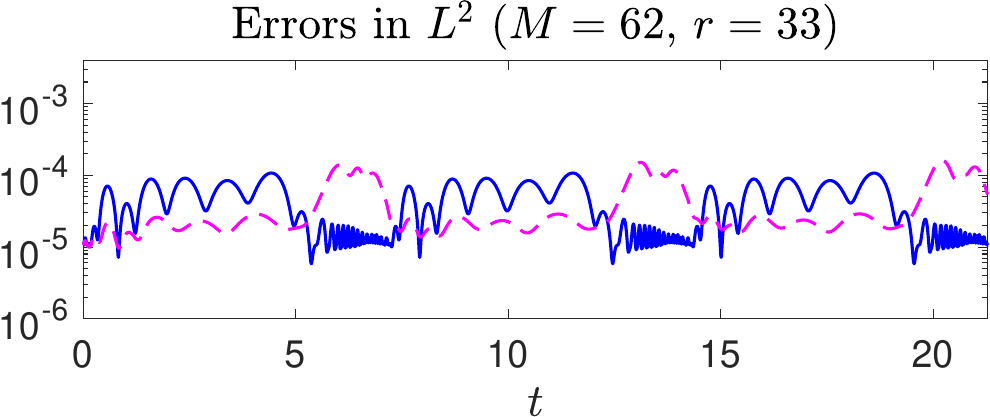}
\begin{caption}{\label{fig_patches} Errors $\left\|\bu_h-P_r\bu_h\right\|_0$ (continuous blue line), $\left\|\bu_r-\bu_h\right\|_0$ (discontinuous line) for~$\Delta t =T/128$ in~$[0,T/32]\cup [(23/32)T,T]$ and, in $[T/32,(23/32)T]$, $\Delta t=T/64$ (left) and $\Delta t= T/32$ (right).  Errors measured every $T/2048$ time units over three periods ($T$ being the period).}
\end{caption}
\end{center}
\end{figure}
We see that the smaller value of~$\Delta t$ on the difficult parts of the time interval decreases the errors almost to the levels of a uniform grid with $\Delta T=T/128$ (left plot in Fig.~4). 
%Further experience with different patches of uniform grids will be reported elsewhere.
{A detailed discussion of further experience with different patches of uniform grids is outside the scope of the present paper and it will be reported elsewhere.}

We now turn our attention to the time derivative case. Here, we found little difference with the finite difference case in terms of the maximum error in time, as it can be checked in Table~\ref{table2}, where we repeat the computations of Table~\ref{table1} for the time derivative case.
\begin{table}
$$
\begin{array}{|r|c|c|c|c|c|}
\hline
\vphantom{\Big|}
M\  & r& \left\| \bu_h-\bu_r\right\|_0 & \left\| \bu_h-P^r\bu_h\right\|_0 & \left\|\nabla (\bu_h-\bu_r)\right\|_0 &  \left\|\nabla (I-P^r)\bu_h \right\|_0 \\ \hline
\vphantom{\big|} 512 &  35&   8.059\times 10^{-5} &    9.985\times 10^{-5}  &   8.596\times 10^{-4}  &   7.195\times 10^{-4} \\ \hline
\vphantom{\big|} 256  &  34&     7.755\times 10^{-5} &    1.000\times 10^{-4}  &   8.587\times 10^{-4}  &   7.180\times 10^{-4} \\ \hline
\vphantom{\big|} 128  & 32&     9.970 \times 10^{-5} &    1.219\times 10^{-4}  &   1.165\times 10^{-3}  &   8.188\times 10^{-4} \\ \hline
\vphantom{\big|} 64  &  23&      4.224\times 10^{-3} &     4.212\times 10^{-3}  &   1.567\times 10^{-1}  &   1.533\times 10^{-1} \\ \hline
\vphantom{\big|} 32  &  17&     1.764\times 10^{-1} &     1.237\times 10^{-1}  &   2.871 &   2.180  \\ \hline
\end{array}
$$
\caption{\label{table2} 
Errors for different values of $M$ measured every $T/2048$ units of time for four periods. Values of $r$ are $r=33$ for $M\ge 128$ and~$r=35$ otherwise. Time derivative case.
}
\end{table}
The values in Table~\ref{table2} are very similar to those in Table~\ref{table1}, except for the smallest values of~$M$ where they are slightly larger
than the corresponding values in~Table~\ref{table1}.
Yet, as Fig.~\ref{fig5} shows, where we show the time derivative case for $M=128$ and $M=64$, there can be some differences when we look at the errors in the whole time interval and not only where they achieve its maximum value. Although the errors for $M=128$ (left plot) are very similar to those on the left plot in~Fig.~\ref{fig3} (only slightly smaller for $\bu_h-P^r\bu_h$), this is not the case for $M=64$ (right plot) where errors, away from where they achieve their peek value,
do not decrease as much as in the finite difference case (left plot in Fig.~\ref{fig4}). For a better comparison, the axes in Fig.~\ref{fig5} are the same as in the corresponding plots in the finite difference case.
 \begin{figure}[h]
 \begin{center}
\includegraphics[height=2.7truecm]{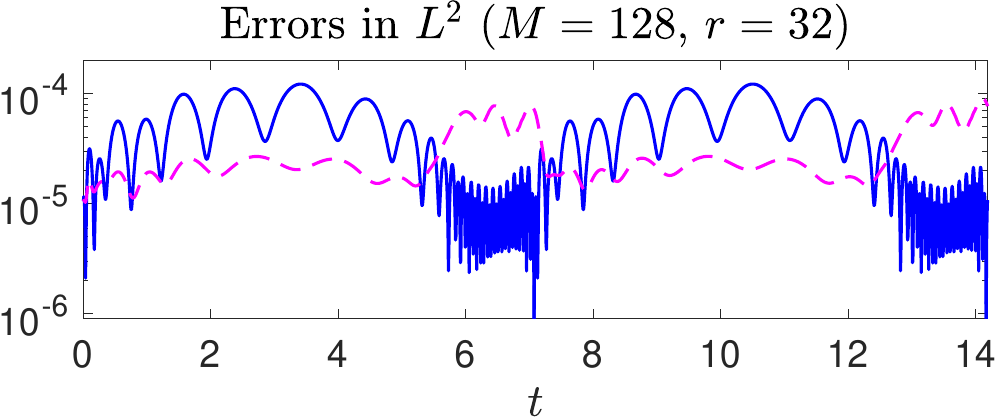}
%\quad
\includegraphics[height=2.7truecm]{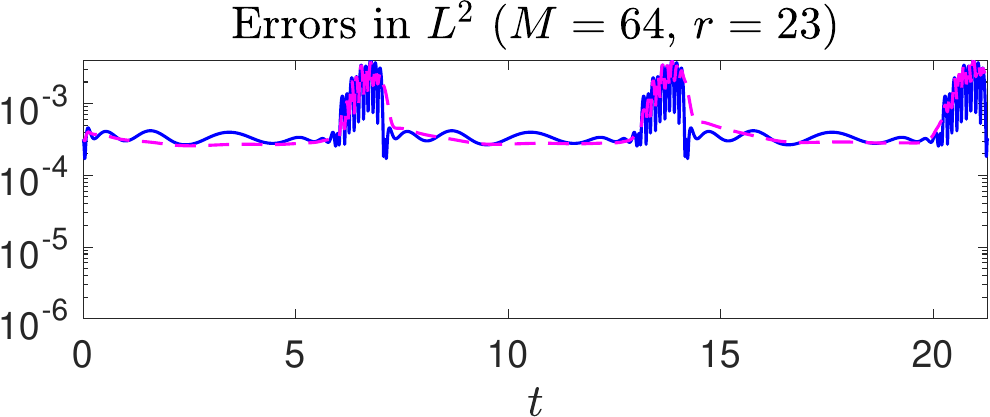}
\begin{caption}{\label{fig5} Errors $\| \bu_h-P^r\bu_h\|_0$ (continuous line) and~$\| \bu_h-\bu_r\|_0$ (discontinuous line) in the time derivative case, for
$M=128$ (left) and $M=64$ (right). Errors measured every $T/2048$ time units over two periods 
{(left) and three periods (right)}, where $T$ is the length of the period.}
\end{caption}
\end{center}
\end{figure}
This may well reflect the presence of extra error terms in the error bounds when time derivatives are used (compare for example~Lemma~\ref{le:maxPr_FD} and
Lemma~\ref{le:maxPr_ut} and \eqref{eq:I_Pr_u_h_dif} and \eqref{eq:I_Pr_u_h_time}). One may wonder why 
%this apparent higher lower bound of the errors
{the impact of the additional terms in the error bounds}
is not seen for $M=128$. The answer is that this is because that lower bound is much smaller
as Fig.~\ref{fig6} shows, where we consider the case $M=128$ but with $r=45$ instead, and where we can see that in the finite difference case the errors, away from peek values, decrease much more than in the time derivative case.
 \begin{figure}[h]
 \begin{center}
\includegraphics[height=2.7truecm]{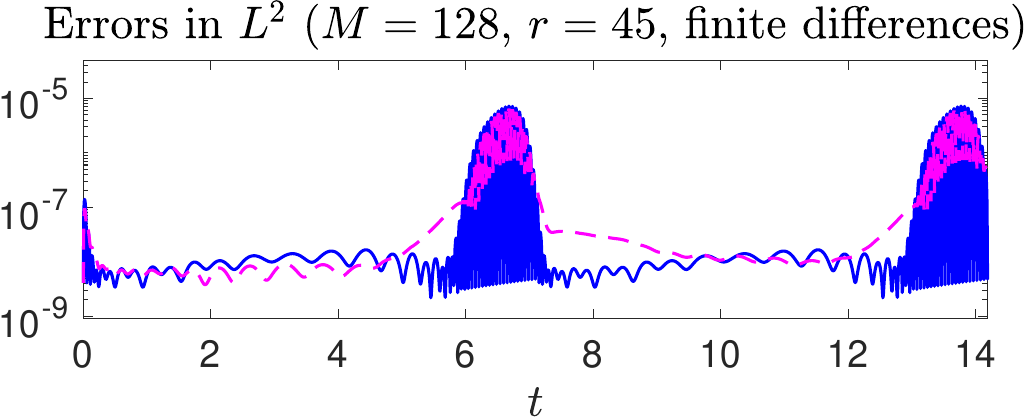}
%\quad
\includegraphics[height=2.7truecm]{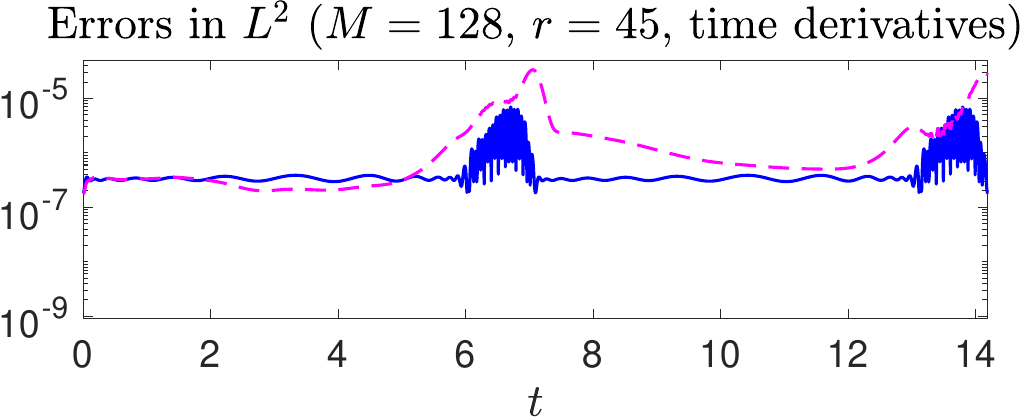}
\begin{caption}{\label{fig6} Errors $\| \bu_h-P^r\bu_h\|_0$ (continuous line) and~$\| \bu_h-\bu_r\|_0$ (discontinuous line) for $M=128$ and $r=45$ in the finite difference case (left plot) and in the time derivative case (right-plot).
Errors measured every $T/2048$ time units over two periods.}
\end{caption}
\end{center}
\end{figure}

\section{Conclusions and future work}

For POD-ROM methods based on snapshots of a FOM  (FEM) approximation at a discrete set of times $(t_n)_{n=0}^M$ in an interval~$[0,T]$, error bounds of the POD-ROM approximation for all $t\in[0,T]$ have been derived in the case of semilinear parabolic equations. The POD-ROM approximation considered is continuous in time, that is, it is the solution of a system of ODEs. The obtained bounds depend on the tail of eigenvalues of the covariance matrix as well as on powers of the distance in time, $t_{n+1}-t_n$, between two consecutive snapshots multiplying some norm of time derivatives of the FEM approximation. The POD basis considered in the present paper is extracted from either first order divided differences of the snapshots or first time derivatives of the FEM approximation together with the snapshot at the initial time (or the mean of the snapshots). This setup allowed to obtain pointwise-in-time error bounds.

In practice, one computes the POD-ROM approximation on a discrete set of times by applying a time integrator to the continuous-in-time POD-ROM discretization.  Typically, this time integrator is the implicit Euler method, but it can be any other one. The analysis in the present paper can be extended to cover the fully-discrete case, allowing for  POD-ROM approximations at times different from those of the snapshots. The analysis of this extension, however, will be carried out in future works.

Numerical studies, first, confirm that POD-ROM methods based on snapshots at a discrete set of times $(t_n)_{n=0}^M$ in an interval~$[0,T]$ obtain accurate approximations for all~$t\in[0,T]$. Second, they show that too much distance between the snapshots degrades the accuracy of the POD-ROM approximation, but beyond a certain (problem-dependent) optimal distance, further decreasing the distance between snapshots does not necessarily improve the accuracy. Third, the numerical studies
suggest that a denser concentration snapshots on those parts of the time interval where time derivatives are large while keeping a smaller concentration elsewhere allows for a significant improvement of the accuracy of the POD-ROM approximation; this also allows for a reduction of the cost of the offline phase. Finally, it turned out that there is not much difference in accuracy of POD-ROM approximations between methods whose POD basis is obtained from divided differences and those where it is computed with time derivatives, although, away from areas where peak errors (or peak values of time derivatives of FOM approximations) are located, the former seem to be more accurate. It remains as topics for future work to find out how to best distribute snapshots to maximize accuracy as well as to choose the optimal number of snapshots.
 \bibliographystyle{abbrv}

\end{document}